\documentclass[12pt]{amsart}
   \usepackage{amsmath,amsthm}
   \usepackage{amsfonts}   
   \usepackage{amssymb}    
   \usepackage{url}
    \usepackage{pstricks}
    \usepackage{graphicx}

\advance \topmargin by -\headheight
\advance \topmargin by -\headsep
      
\evensidemargin \oddsidemargin
\usepackage[all]{xy}

\newtheorem{thm}{Theorem}[section]
\newtheorem{prop}[thm]{Proposition}
\newtheorem{lem}[thm]{Lemma}
\newtheorem{cor}[thm]{Corollary}

\newtheorem{defn}[thm]{Definition}
\newtheorem{qu}[thm]{Question}

\newtheorem{thmx}{Theorem}

\theoremstyle{remark}
\newtheorem{rem}[thm]{Remark}
\newtheorem{ex}[thm]{Example}
\newtheorem{constr}[thm]{Construction}

\title{Contact structures, deformations and taut foliations}
\author{Jonathan Bowden}
\address{Mathematisches Institut, Ludwig-Maximillians-Universit\"at, Theresienstr. 39, 80333 Munich, Germany}
\email{jonathan.bowden@math.lmu.de}
\date{\today}
\begin{document}
\maketitle
\begin{abstract}
Using deformations of foliations to contact structures as well as rigidity properties of Anosov foliations we provide infinite families of examples which show that the space of taut foliations in a given homotopy class of plane fields need not be path connected. Similar methods also show that the space of representations of the fundamental group of a hyperbolic surface to the group of smooth diffeomorphisms of the circle with fixed Euler class is in general not path connected. As an important step along the way we resolve the question of which universally tight contact structures on Seifert fibered spaces are deformations of taut or Reebless foliations when the genus of the base is positive or the twisting number of the contact structure in the sense of Giroux is non-negative.
\end{abstract}

\section{Introduction}
In their book on confoliations Eliashberg and Thurston \cite{ETh} established a fundamental link between the theory of foliations and contact topology, by showing that any foliation that is not the product foliation on $S^2 \times S^1$ can be $C^0$-approximated by a contact structure. The proof of this result naturally leads to the study of confoliations, which are a generalisation of both contact structures and foliations. Recall that a smooth cooriented 2-plane field $\xi = \text{Ker}(\alpha)$ on an oriented 3-manifold $M$ is a confoliation if $\alpha \wedge d \alpha \geq 0$. For the most part interest has focussed on the contact case, where the study of deformations and isotopy are equivalent in view of Gray's Stability Theorem. On the other hand many questions in the deformation theory of foliations or more generally confoliations remain to a large extent unexplored. 

Rather than considering general confoliations, we will focus on questions concerning the topology of the space of foliations. In contact topology one has a tight vs.\ overtwisted dichotomy, which is in some sense mirrored in the theory of foliations by Reebless foliations and those with Reeb components. In analogy with 3-dimensional contact topology where one seeks to understand deformation classes of tight contact structures, we will be primarily concerned with studying the topology of the space of Reebless and taut foliations and the contact structures approximating them. 




It is well known that every contact structure is isotopic to a deformation of a foliation by Etnyre \cite{Etn} (see also Mori \cite{Mor}). More precisely, Etnyre showed that for any contact structure $\xi$ there is a smooth $1$-parameter family $\xi_t$ such that $\xi_0$ is integrable and $\xi_t$ is a contact structure isotopic to $\xi$ for $t >0$. The foliations that Etnyre considers are constructed by completing the foliation given by the pages of an open book supporting the contact structure $\xi$ to a genuine foliation by inserting Reeb components in a neighbourhood of the binding and spiralling accordingly. This led Etnyre to ask whether every universally tight contact structure on a manifold with infinite fundamental group is a deformation of a Reebless foliation. By considering the known criteria for the existence of Reebless foliations on small Seifert fibered spaces, it is easy to see that this is false in general. This was first observed by Lekili and Ozbagci \cite{LeO}. Nevertheless it is still an interesting problem to determine which contact structures can be realised as deformations of Reebless foliations, a problem which was already raised by Eliashberg and Thurston in \cite{ETh}. Furthermore, the counter examples coming from small Seifert fibered spaces are not completely satisfactory, since the obvious necessary condition for a manifold to admit a Reebless foliation is that it admits universally tight contact structures for both orientations, and for small Seifert manifolds this is in fact equivalent to the existence of a Reebless foliation (cf.\ Proposition \ref{existence_crit}).

In contrast to the case of small Seifert manifolds Etnyre's original question has a positive answer for Seifert fibered spaces whose bases have positive genus. There are two cases depending on whether the twisting number $t(\xi)$ of the contact structure $\xi$ is positive or not (cf.\ Definition \ref{def_twisting}). 
\begin{thmx}\label{def_Int}
Let $\xi$ be a universally tight contact structure on a Seifert fibered space with infinite fundamental group and $t(\xi) \geq 0$, then $\xi$ is isotopic to a deformation of a Reebless foliation. If $g>0$ and $t(\xi) < 0$, then $\xi$ is isotopic to a deformation of a taut foliation.
\end{thmx}
\noindent The proof of Theorem \ref{def_Int} involves examining the Giroux-type normal forms for universally tight contact structures of Massot \cite{Mas}, \cite{Mas2} and considering foliations that are well adapted to these normal forms. The cases of negative and non-negative twisting are treated separately, with the former being reduced to the $t(\xi)=-1$ case via a covering trick.

Other examples of tight contact structures on manifolds that admit no Reebless foliations were given by Etg\"u \cite{Etg} in the case that the manifold is not Seifert fibered but hyperbolic. However, these examples are neither known to be universally tight, nor is it shown that there are tight contact structures for both orientations. This then suggests the following refinement of Etnyre's original question, which then has an affirmative answer for Seifert fibered spaces:
\begin{qu}\label{taut_def_1}
Does every irreducible $3$-manifold with infinite fundamental group that admits both positive and negative universally tight contact structures necessarily admit a (smooth) Reebless foliation?
\end{qu}




Until recently there was little known about the topology of the space of foliations on a 3-manifold. For the class of horizontal foliations on $S^1$-bundles Larcanch\'e \cite{Lar} showed that the inclusion of the space of horizontal integrable plane fields into the space of all integrable plane fields is homotopic to a point and in particular its image is contained in a single path component in the space of all integrable plane fields. She also showed that any integrable plane field that is sufficiently close to the tangent distribution $T\mathcal{F}$ of a taut foliation $\mathcal{F}$ can be deformed to $T\mathcal{F}$ through integrable plane fields. In her PhD thesis Eynard-Bontemps showed that a much more general result holds. In particular, she proved the following theorem, which mirrors Eliashberg's $h$-principal for overtwisted contact structures.
\begin{thm}[Eynard-Bontemps \cite{Eyn}]
Let $\mathcal{F}_0$ and $\mathcal{F}_1$ be smooth oriented taut foliations on a $3$-manifold $M$ whose tangent distributions are homotopic as (oriented) plane fields. Then $T\mathcal{F}_0$ and $T\mathcal{F}_1$ are smoothly homotopic through integrable plane fields.
\end{thm}
\noindent The foliations that Eynard-Bontemps constructs use a parametric version of a construction of Thurston, first exploited by Larcanch\'e, that allows foliations to be extended over solid tori using foliations that contain Reeb components. In view of this it is natural to ask whether any two horizontal foliations are in fact homotopic through horizontal foliations or more generally whether any two taut foliations whose tangent plane fields are homotopic are homotopic through taut or even Reebless foliations. Since any horizontal foliation on an $S^1$-bundle is essentially determined by its holonomy representation, the former question is then related to the topology of the representation space $\text{Rep}(\pi_1(\Sigma_g), \text{Diff}_+(S^1))$ considered with its natural $C^{\infty}$-topology. Concerning the topology of this space we prove the following:
\begin{thmx}\label{int_Comp}
Let $\#Comp(e)$ denote the number of path components of $\text{Rep}_e(\pi_1(\Sigma_g), \text{Diff}_+(S^1))$ with fixed Euler class $e \neq 0$ such that $e$ divides $2g-2 \neq 0$ and write $2g-2 = n\thinspace e$. Then the following holds:
\[\#Comp(e) \geq n^{2g} +1.\]
\end{thmx}
This should be compared with results of Goldman \cite{Gol} who showed that the space of representations of $\pi_1(\Sigma_g)$ to the $n$-fold cover of $\textit{PSL}(2,\mathbb{R})$ with Euler number satisfying $2g-2 = n\thinspace e$ has precisely $n^{2g}$ components. In particular, the proof of Theorem \ref{int_Comp} shows that the images of these components in $\text{Rep}_e(\pi_1(\Sigma_g), \text{Diff}_+(S^1))$ remain distinct. The idea behind the proof of this theorem is very simple: a smooth family of representations $\rho_t$ corresponds to a smooth family of foliations $\mathcal{F}_t$ via the suspension construction and one then deforms this family to a family of contact structures using a parametric version of Eliashberg and Thurston's perturbation theorem. Deformations of contact structures correspond to isotopies via Gray's Stability Theorem and this then gives an isotopy of contact structures, which then distinguish path components in the representation space.

In general, however, there is no parametric version of Eliashberg and Thurston's perturbation theorem, since in general the contact structure approximating a foliation is not unique. On the other hand under certain additional assumptions, that are for instance true for horizontal foliations on non-trivial $S^1$-bundles, Vogel has shown a remarkable uniqueness result for the isotopy class of a contact structure approximating a foliation, which implies in particular that this isotopy class is in fact a $C^0$-deformation invariant in certain situations. 
\begin{thm}[Vogel \cite{Vog}]
Let $\mathcal{F}$ be an oriented $C^2$-foliation without torus leaves. Assume furthermore that $\mathcal{F}$ is neither a foliation by planes nor by cylinders only. Then there is a $C^0$-neighbourhood $\mathcal{U}_0$ of $T\mathcal{F}$ in the space of oriented plane fields so that all positive contact structures in $\mathcal{U}_0$ are isotopic.
\end{thm}
\noindent If one considers only deformations that are $C^{\infty}$, or continuous in the $C^1$-topology would even suffice, then one can give a comparatively simple argument using linear perturbations to deform families of foliations to contact structures in a smooth manner (cf.\ Section \ref{perturb_fol}). Such deformations then provide the desired obstructions used to prove Theorem \ref{int_Comp} and this suffices for our purposes.


We also present a second independent proof of Theorem \ref{int_Comp}, which uses the rich structure theory of Anosov foliations instead of contact topology. Matsumoto \cite{Mat} has shown that any representation in $\text{Rep}_e(\pi_1(\Sigma_g), \text{Diff}_+(S^1))$ with maximal Euler class $e = \pm(2g-2)$ is topologically conjugate to a Fuchsian representation given by a cocompact lattice in $\textit{PSL}(2,\mathbb{R})$ and Ghys \cite{Ghy} showed that this conjugacy can be assumed to be smooth. Furthermore, the space of Fuchsian representations of $\pi_1(\Sigma_g)$ can be identified with Teichm\"uller space and is thus contractible. The suspension foliations corresponding to Fuchsian representations are Anosov in the sense that they are diffeomorphic to the weak stable foliation of the Anosov flow given by the geodesic flow of some hyperbolic metric on the unit cotangent bundle $ST^*\Sigma_g$. By considering fiberwise coverings it is easy to construct Anosov representations with non-maximal Euler classes. In general not every horizontal foliation lies in the same component as an Anosov foliation. We do however obtain the following analogue of Ghys' result, which answers a question posed to us by Y. Mitsumatsu.
\begin{thmx}
Any representation $\rho \in \text{Rep}(\pi_1(\Sigma_g),\text{Diff}_+(S^1))$ that lies in the $C^0$-path component of an Anosov representation $\rho_{An}$ is itself Anosov. In particular, it is conjugate to a discrete subgroup of a finite covering of $PSL(2,\mathbb{R})$ and is injective.
\end{thmx}
\noindent Since non-injective representations always exist in the case of non-maximal Euler class this immediately implies the existence of more than one path component in the representation space for any non-maximal Euler class that admits Anosov representations. By using certain conjugacy invariants (cf.\ Theorem \ref{open_closed_Anosov}) it is then easy to recover the precise estimates of Theorem \ref{int_Comp}.

Of course not every taut foliation on an $S^1$-bundle is horizontal so this theorem still leaves open the question of whether taut foliations are always deformable through taut foliations. The first example of a pair of oriented taut foliations that are homotopic as foliations but not as (oriented) taut foliations is due to Vogel \cite{Vog}. By considering foliations on certain small Seifert fibered manifolds we obtain an infinite family of examples that have the additional properties that they still cannot be deformed to one another through taut foliations even if one forgets orientations or if one considers only diffeomorphism classes of foliations, in contrast to the situation in Vogel's example. 
\begin{thmx}\label{taut_not_conn}
There exist  an infinite family of manifolds $M_n$ each admitting a pair of taut foliations $\mathcal{F}_0,\mathcal{F}_1$ that are homotopic as oriented foliations but not as taut foliations. Furthermore, the same result holds true for unoriented foliations or if one considers diffeomorphism classes of foliations.
\end{thmx}
\noindent Since the manifolds considered in Theorem \ref{taut_not_conn} are non-Haken the notions of tautness and Reeblessness coincide, so in particular it follows that any deformation of foliations between the foliations $\mathcal{F}_0$ and $\mathcal{F}_1$ must contain Reeb components.

Further examples of taut foliations that cannot be joined by a path in the space of Reebless foliations are given by using the special structure of foliations on the unit cotangent bundle over a closed surface of genus at least $2$. In particular, we show that the weak unstable foliation of the geodesic flow $\mathcal{F}_{hor}$ on $ST^*\Sigma_g$ cannot be smoothly deformed to any taut foliation with a torus leaf $\mathcal{F}_T$ without introducing Reeb components (Corollary \ref{Reebless_comp}). One can view this fact as a generalisation of the result of Ghys and Matsumoto concerning horizontal foliations of $ST^*\Sigma_g$, in that it shows that the path component of an Anosov foliation in the space of all Reebless foliations contains only Anosov foliations. This is perhaps slightly surprising since for the product foliation on $\Sigma_g \times S^1$ one can spiral along any vertical torus $\gamma \times S^1$ to obtain smooth deformations that introduce incompressible torus leaves. On the other hand, although there exists no smooth deformation through taut foliations, one can construct a taut deformation between $\mathcal{F}_{hor}$ and $\mathcal{F}_T$ through foliations that are only of class $C^0$ (Proposition \ref{C_0_difference}). Thus these examples highlight once more the difference between foliations of class $C^0$ and those of higher regularity.



\subsection*{Outline of paper:} In Section \ref{Fol_and_Cont} we recall some basic definitions and constructions of foliations and contact structures and in Section \ref{Seifert_man} we review some basic facts about Seifert fibered spaces and horizontal foliations. Section \ref{perturb_fol} contains the relevant versions of Eliashberg and Thurston's results on deforming foliations to contact structures and Section \ref{hor_contact} contains background on horizontal contact structures and normal forms. In Sections \ref{Deform_taut} and \ref{Deform_Reebless} we prove Theorem \ref{def_Int} first for negative twisting numbers and then in the non-negative case. Section \ref{horizontal_and_taut} contains our main results concerning deformations of taut foliations and finally in Section \ref{Anosov_fol} we analyse components of the representation space of a surface group that contain Anosov representations, yielding an alternative proof of Theorem \ref{int_Comp}.

\subsection*{Acknowledgments:} We thank T. Vogel for his patience in explaining many wonderful ideas which provided the chief source of inspiration for the results of this article. We also thank Y. Mitsumatsu for his stimulating questions and H. Eynard-Bontemps, S. Matsumoto and the referee for helpful comments. The hospitality of the Max Planck Institute f\"ur Mathematik in Bonn, where part of this research was carried out, is also gratefully acknowledged. This research was also partially supported by DFG Grant BO4423/1-1.

\subsection*{Conventions:} Unless otherwise specified all manifolds, contact structures and foliations are smooth and (co)oriented and all manifolds are closed.

\section{Foliations and contact structures}\label{Fol_and_Cont}
In this section we recall some basic definitions and constructions for foliations and contact structures. For a more in depth discussion of foliations on $3$-manifolds we refer to the book of Calegari \cite{Cal}.

\subsection{Foliations:} A codimension-$1$ foliation $\mathcal{F}$ on a $3$-manifold $M$ is a decomposition of $M$ into connected injectively immersed surfaces called \emph{leaves} that is locally diffeomorphic to level sets of the projection of $\mathbb{R}^3$ to the $z$-axis. We will always assume that all foliations are smooth and cooriented unless otherwise specified. One can then define a global non-vanishing $1$-form $\alpha$ by requiring that 
$$\text{Ker}(\alpha) = T \mathcal{F} = \xi \subset T M.$$
By Frobenius' Theorem a cooriented distribution is tangent to a foliation if and only if 
$$\alpha \wedge d \alpha \equiv 0$$
and in this case $\xi$ is called \emph{integrable}. An important example of a foliation is the Reeb foliation.
\begin{ex}[Reeb foliation]
Consider $D^2 \times S^1$ with coordinates $\left((r,\theta), \phi\right)$. Choose a non-negative function $\gamma(r)$ on $[0,1]$ that is infinitely tangent to a constant map at the end points, is decreasing on the interior and has $\gamma(0) = 1,\gamma(1) = 0$. Then $\mathcal{F}_{Reeb}$ is defined as the kernel of the following form
$$\alpha = \gamma(r) \thinspace d \phi + (1-\gamma(r)) dr.$$
This foliation has a unique compact leaf given by $\partial D^2 \times S^1$ and the foliation on $int(D^2) \times S^1$ is by parabolic planes. A solid torus with such a foliation will be called a \emph{Reeb component}.
\end{ex}
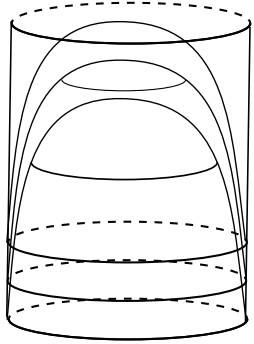
\begin{figure} 
\psset{xunit=.2pt,yunit=.2pt,runit=.2pt}
\begin{pspicture}(464.16421509,641.04284668)
{
\newrgbcolor{curcolor}{0 0 0}
\pscustom[linewidth=2.5999999,linecolor=curcolor]
{
\newpath
\moveto(2.21423,39.09284668)
\curveto(17.6073,222.67165668)(25.67513,394.17312668)(138.44584,443.53023668)
}
}
{
\newrgbcolor{curcolor}{0 0 0}
\pscustom[linewidth=2.5999999,linecolor=curcolor]
{
\newpath
\moveto(138.44584,443.53023668)
\curveto(160.2713,453.08272668)(186.01862,458.06004668)(216.49994,457.66427668)
\curveto(247.84707,457.63477668)(274.52875,452.93394668)(297.32944,444.25849668)
}
}
{
\newrgbcolor{curcolor}{0 0 0}
\pscustom[linewidth=2.5999999,linecolor=curcolor]
{
\newpath
\moveto(297.32944,444.25849668)
\curveto(426.21492,395.21879668)(431.0907,219.17672668)(453.6428,41.94998668)
\lineto(453.6428,41.94998668)
\lineto(453.6428,41.94998668)
\lineto(453.6428,41.94998668)
\lineto(453.6428,41.94998668)
\lineto(453.6428,41.94998668)
}
}
{
\newrgbcolor{curcolor}{0 0 0}
\pscustom[linewidth=3.9000001,linecolor=curcolor,linestyle=dashed,dash=15.60000038 15.60000038]
{
\newpath
\moveto(2.9285471,113.37850668)
\curveto(2.9285471,134.68091668)(103.9842671,151.94993668)(228.6428171,151.94993668)
\curveto(353.3013771,151.94993668)(454.3570971,134.68091668)(454.3570971,113.37850668)
}
}
{
\newrgbcolor{curcolor}{0 0 0}
\pscustom[linewidth=3.9000001,linecolor=curcolor]
{
\newpath
\moveto(454.3570971,113.37850668)
\curveto(454.3570971,92.07608668)(353.3013771,74.80707668)(228.6428171,74.80707668)
\curveto(103.9842671,74.80707668)(2.9285471,92.07608668)(2.9285471,113.37850668)
}
}
{
\newrgbcolor{curcolor}{0 0 0}
\pscustom[linewidth=2.5999999,linecolor=curcolor]
{
\newpath
\moveto(2.92853,111.94993668)
\curveto(21.30075,331.05827668)(29.23791,532.96205668)(217.21424,530.52136668)
\curveto(425.7573,530.32481668)(427.81539,323.38638668)(454.3571,114.80707668)
\lineto(454.3571,114.80707668)
\lineto(454.3571,114.80707668)
\lineto(454.3571,114.80707668)
\lineto(454.3571,114.80707668)
\lineto(454.3571,114.80707668)
}
}
{
\newrgbcolor{curcolor}{0 0 0}
\pscustom[linewidth=3.9000001,linecolor=curcolor,linestyle=dashed,dash=15.60000038 15.60000038]
{
\newpath
\moveto(4.357094,186.23558668)
\curveto(4.357094,207.53799668)(105.412814,224.80701668)(230.071364,224.80701668)
\curveto(354.729924,224.80701668)(455.785644,207.53799668)(455.785644,186.23558668)
}
}
{
\newrgbcolor{curcolor}{0 0 0}
\pscustom[linewidth=3.9000001,linecolor=curcolor]
{
\newpath
\moveto(455.785644,186.23558668)
\curveto(455.785644,164.93316668)(354.729924,147.66415668)(230.071364,147.66415668)
\curveto(105.412814,147.66415668)(4.357094,164.93316668)(4.357094,186.23558668)
}
}
{
\newrgbcolor{curcolor}{0 0 0}
\pscustom[linewidth=2.5999999,linecolor=curcolor]
{
\newpath
\moveto(1.499934,184.80701668)
\curveto(19.872154,403.91535668)(27.809314,605.81913378)(215.785644,603.37844378)
\curveto(424.328704,603.18189168)(426.386794,396.24346668)(452.928504,187.66415668)
\lineto(452.928504,187.66415668)
\lineto(452.928504,187.66415668)
\lineto(452.928504,187.66415668)
\lineto(452.928504,187.66415668)
\lineto(452.928504,187.66415668)
}
}
{
\newrgbcolor{curcolor}{0 0 0}
\pscustom[linewidth=3,linecolor=curcolor]
{
\newpath
\moveto(1.5,44.80713668)
\lineto(10.07143,604.80713668)
\lineto(10.07143,604.80713668)
}
}
{
\newrgbcolor{curcolor}{0 0 0}
\pscustom[linewidth=3.9000001,linecolor=curcolor,linestyle=dashed,dash=15.60000038 15.60000038]
{
\newpath
\moveto(5.071394,40.52141668)
\curveto(5.071394,61.82382668)(106.127114,79.09284668)(230.785664,79.09284668)
\curveto(355.444224,79.09284668)(456.499944,61.82382668)(456.499944,40.52141668)
}
}
{
\newrgbcolor{curcolor}{0 0 0}
\pscustom[linewidth=3.9000001,linecolor=curcolor]
{
\newpath
\moveto(456.499944,40.52141668)
\curveto(456.499944,19.21899668)(355.444224,1.94998668)(230.785664,1.94998668)
\curveto(106.127114,1.94998668)(5.071394,19.21899668)(5.071394,40.52141668)
}
}
{
\newrgbcolor{curcolor}{0 0 0}
\pscustom[linewidth=3.9000001,linecolor=curcolor]
{
\newpath
\moveto(456.499944,40.52141668)
\curveto(456.499944,19.21899668)(355.444224,1.94998668)(230.785664,1.94998668)
\curveto(106.127114,1.94998668)(5.071394,19.21899668)(5.071394,40.52141668)
}
}
{
\newrgbcolor{curcolor}{0 0 0}
\pscustom[linewidth=3.9000001,linecolor=curcolor,linestyle=dashed,dash=15.60000038 15.60000038]
{
\newpath
\moveto(10.78568,600.52141668)
\curveto(10.78568,621.82382668)(111.8414,639.09284668)(236.49995,639.09284668)
\curveto(361.15851,639.09284668)(462.21423,621.82382668)(462.21423,600.52141668)
}
}
{
\newrgbcolor{curcolor}{0 0 0}
\pscustom[linewidth=3.9000001,linecolor=curcolor]
{
\newpath
\moveto(462.21423,600.52141668)
\curveto(462.21423,579.21899668)(361.15851,561.94998668)(236.49995,561.94998668)
\curveto(111.8414,561.94998668)(10.78568,579.21899668)(10.78568,600.52141668)
}
}
{
\newrgbcolor{curcolor}{0 0 0}
\pscustom[linewidth=3.9000001,linecolor=curcolor]
{
\newpath
\moveto(462.21423,600.52141668)
\curveto(462.21423,579.21899668)(361.15851,561.94998668)(236.49995,561.94998668)
\curveto(111.8414,561.94998668)(10.78568,579.21899668)(10.78568,600.52141668)
}
}
{
\newrgbcolor{curcolor}{0 0 0}
\pscustom[linewidth=3,linecolor=curcolor]
{
\newpath
\moveto(452.92857,33.37856668)
\lineto(461.5,593.37856668)
\lineto(461.5,593.37856668)
}
}
{
\newrgbcolor{curcolor}{0 0 0}
\pscustom[linewidth=3.19023731,linecolor=curcolor]
{
\newpath
\moveto(399.78496043,337.60051689)
\curveto(399.78496043,319.33460284)(320.92343417,304.52716536)(223.64280223,304.52716536)
\curveto(126.36217809,304.52716536)(47.50065184,319.33460284)(47.50065184,337.60051689)
}
}
{
\newrgbcolor{curcolor}{0 0 0}
\pscustom[linewidth=2.07433951,linecolor=curcolor]
{
\newpath
\moveto(339.91015581,493.56169332)
\curveto(339.91015581,481.9545679)(287.44202296,472.54514094)(222.71929486,472.54514094)
\curveto(157.99657196,472.54514094)(105.52843911,481.9545679)(105.52843911,493.56169332)
}
}
\end{pspicture} \caption{A piece of the Reeb foliation}\end{figure}
General foliations are very flexible - they satisfy an $h$-principal due to Wood \cite{Wood} - and in particular every $2$-plane field is homotopic to the tangent distribution of a foliation. A more geometrically significant class of foliations are those that are \emph{taut}. Here a foliation is taut if every leaf admits a closed transversal. Note that any foliation that contains a Reeb component is not taut, since the boundary leaf of the Reeb component is separating and compact. Thus taut foliations fall into the more general class of \emph{Reebless} foliations, i.e.\ those that contain no Reeb component. The existence of a Reebless foliation puts restrictions on the topology of $M$ due to the following theorem that is usually attributed Novikov, although the statement about incompressibility may be due to Thurston.
\begin{thm}[Novikov]
Let $\mathcal{F}$ be a Reebless foliation on a $3$-manifold. Then all leaves of $\mathcal{F}$ are incompressible, $\pi_2(M) = 0$ and all transverse loops are essential in $\pi_1(M)$. In particular, $\pi_1(M)$ is infinite.
\end{thm}
\noindent It follows from Novikov's theorem that a foliation is Reebless if and only if all its torus leaves are incompressible. We also have the following criterion for tautness, which follows from Novikov's notion of dead end components combined with the Poincar\'e-Hopf Theorem.
\begin{thm}\label{Goodman}
Let $\mathcal{F}$ be a foliation on a $3$-manifold $M$. If no oriented combination of torus leaves of $\mathcal{F}$ is null-homologous in $H_2(M)$, then $\mathcal{F}$ is taut.
\end{thm}
It will be important to modify foliations in various situations below and we will repeatedly make use of a spinning construction which introduces toral leaves into foliations that are transverse to an embedded torus.
\begin{constr}[Spiralling along a torus]\label{Spiral}
Let $\mathcal{F}$ be a foliation on a manifold  obtained by cutting a closed manifold $M$ open along an embedded torus
$$\overline{M} =  M \setminus T^2 \times (-\epsilon,\epsilon)$$
and assume that $\mathcal{F}$ is transverse on the boundary components $T_{-},T_{+}$ of $\overline{M}$. We furthermore assume that $\mathcal{F}$ is linear on the boundary so that it is given as the kernel of closed $1$-forms $\alpha_{-}$ and $\alpha_{+}$ respectively. Letting $z$ be the normal coordinate on $T^2 \times (-\epsilon,\epsilon)$ we then define a foliation as the kernel of the following $1$-form
\[\alpha = \rho(-z)\alpha_{-} + \rho(z)\alpha_+ + (1-\rho(|z|))dz.\]
Here $\rho$ is a non-decreasing function that is positive for $z > 0$, satisfies $\rho(z) = 1$ near $\epsilon$ and is identically zero otherwise so that $\rho$ vanishes to infinite order at the origin. 
\end{constr}
\begin{figure}\input{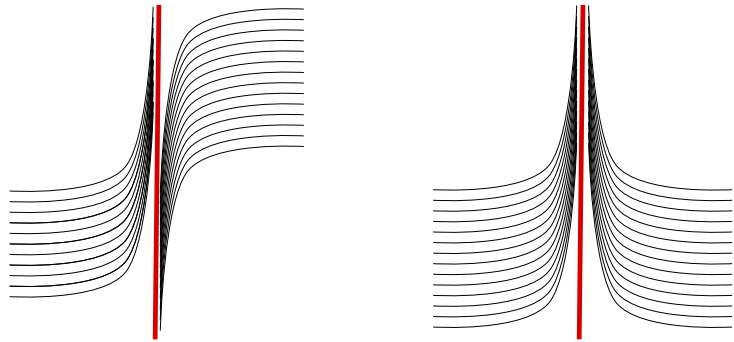}\caption{A cross-section of the foliation obtained after spiralling. The left hand figure shows the unstable case $\alpha_+ = \alpha_-$ and the right hand figure shows the example $\alpha_+ = - \alpha_-$, which is then stable. The torus leaf is represented by the thick (red) line.}\end{figure}
\noindent Note that spiralling along an embedded torus $T$ has the effect of introducing a closed torus leaf. Furthermore, if we consider the foliation given by cutting open a manifold along an embedded torus transverse to a foliation $\mathcal{F}$ such that the induced foliation on $T$ is linear, then we take $\alpha_{+} = \alpha_{-}$ so that foliation obtained by spiralling can be obtained through a smooth 1-parameter deformation of foliations. In this case we will say that the resulting foliation is obtained from $\mathcal{F}$ by \emph{spinning} along the torus $T$. If the induced foliation $\mathcal{F}|_T$ on $T$ is not linear and is without $2$-dimensional Reeb components, then one can still spin along $T$ to obtain a foliation that is only of class $C^0$ as long as we choose the direction of spinning to be transverse to $\mathcal{F}|_T$. Finally observe that if $T$ is a compressible torus given as the boundary of a tubular neighbourhood of a closed transversal, then spiralling along $T$ has the effect of introducing a Reeb component having $T$ as a closed leaf. In this case spinning along $T$ corresponds to \emph{turbulisation}. This in particular shows that Reeblessness and hence tautness are not deformation invariants of foliations.

\subsection{Contact structures:} In addition to foliations we will also consider totally non-integrable plane fields or \emph{contact structures}. Here a contact structure $\xi$ is a distribution such that $\alpha \wedge d \alpha$ is nowhere zero for any defining $1$-form with $\xi = \text{Ker}(\alpha)$. Unless specified our contact structures will always be \emph{positive} with respect to the orientation on $M$ so that $\alpha \wedge d \alpha >0$. If $\alpha$ only satisfies the weaker inequality $\alpha \wedge d \alpha \geq 0$, then $\xi$ is called a (positive) \emph{confoliation}.

There is a fundamental classification of contact structures into those that are tight and those that are not.
\begin{defn}[Overtwistedness]
A contact structure $\xi$ on manifold $M$ is called \textbf{overtwisted} if it admits an embedded disc $D \hookrightarrow M$ such that
$$TD|_{\partial D} = \xi |_{\partial D}.$$
\end{defn}
\noindent If a contact structure $\xi$ admits no such disc then it is called \emph{tight}. A contact structure is \emph{universally tight} if its pullback to the universal cover $\widetilde{M} \to M$ is tight.

\subsection{Topology on the space of plane fields:}
We wish to approximate foliations by contact structures. For this we consider plane fields as sections of the oriented Grassmann bundle of $M$, which can be identified with the unit cotangent bundle $ST^*M$ after a choice of metric. We then say that two plane fields are $C^k$-close, if they $C^k$-close as sections of this bundle. In the context of approximating foliations by contact structures it is most natural to consider the tangent distribution of a foliation rather than the foliation itself. In view of this we will speak about convergence of sequence of foliations $\{\mathcal{F}_n\}$ in the strong sense that $T \mathcal{F}_n$ converges in the $C^k$-topology. In particular, a $C^0$-foliation will be a foliation that is tangent to a continuous $2$-plane field.


\section{Seifert manifolds and horizontal foliations}\label{Seifert_man}
\subsection{Seifert manifolds:}
A Seifert manifold is a closed 3-manifold that admits a locally free $S^1$-action. These manifolds are well understood and can all be built using the following recipe: Let $R$ be an oriented, compact, connected surface (with boundary) of genus $g$ and let $R_i = \partial_i R$ for $ 0 \leq i \leq r$ denote its oriented boundary components. We then obtain a Seifert manifold by gluing solid tori $W_i = D^2 \times S^1$ to the $i$-th boundary component of $R \times S^1$ in such a way that the oriented meridian $m_i = \partial D^2$ maps to $-\alpha_i[R_i] + \beta_i[S^1]$ in homology, where $S^1$ is oriented to intersect $R$ positively and $\alpha_i \neq 0$. 

The obvious $S^1$-action on $R\times S^1$ extends to a locally free $S^1$-action on $M$ in a natural way and the numbers $(g,\frac{\beta_0}{\alpha_0},...,\frac{\beta_r}{\alpha_r})$ are called the Seifert invariants of $M$. This $S^1$-action has a finite number of orbits that have non-trivial stabilisers, which are called \emph{exceptional fibers}. These exceptional fibers correspond to the cores of those solid tori $W_i$ for which the attaching slope $\frac{\beta_i}{\alpha_i}$ is not integral. The Seifert invariants are not unique, as one can add and subtract integers so that the sum $\sum \frac{\beta_i}{\alpha_i}$ remains unchanged to obtain equivalent manifolds. This then corresponds to a different choice of section on $R \times S^1$ with respect to which the Seifert invariants were defined. However, the Seifert invariants can be put in a normal form by requiring that $b = \frac{\beta_0}{\alpha_0} \in \mathbb{Z}$ and that
$$0 < \frac{\beta_1}{\alpha_1} \leq \frac{\beta_2}{\alpha_2} \leq ... \leq \frac{\beta_r}{\alpha_r} < 1.$$
This normal form is then unique, except for a small list of manifolds (see \cite{Hat}). Note that according to our conventions a Seifert fibered space $M$ with normalised Seifert invariants $(g,b,\frac{\beta_1}{\alpha_1},...,\frac{\beta_r}{\alpha_r})$ is an oriented manifold. The Seifert fibered space $\overline{M}$ considered with the opposite orientation has normalised Seifert invariants $ (g,-b-r,1 - \frac{\beta_1}{\alpha_1},...,1 - \frac{\beta_r}{\alpha_r})$.

\medskip

\noindent \textbf{Warning:} The conventions for Seifert manifolds differ greatly in the literature. Here we follow the conventions of \cite{Mas} and \cite{EHN}, which differ from those of \cite{Hon} and \cite{LiM}. 

Given a Seifert manifold $M$ there is a natural fiberwise branched $n$-fold covering given by quotienting out the $n$-th roots of unity $\mathbb{Z}_n \subset S^1$. The Seifert invariants of the quotient manifold can then be easily determined in terms of those of $M$ and we note this in the following proposition for future reference.
\begin{prop}[Fiberwise branched covers]\label{fibre_cov}
Let $M$ be a Seifert manifold with Seifert invariants $(g,b,\frac{\beta_1}{\alpha_1},...,\frac{\beta_r}{\alpha_r})$, where $\alpha_i,\beta_i$ are coprime. Then there is a fiber preserving branched $n$-fold covering map $M \stackrel{p} \longrightarrow M'$, where the quotient space $M'$ has (unnormalised) Seifert invariants $(g,nb,\frac{n\beta_1}{\alpha_1},...,\frac{n\beta_r}{\alpha_r})$. The branching locus of $p$ is a (possibly empty) subset of the exceptional fibers and the branching order around the $i$-th singular fiber is $gcd(n,\alpha_i)$.
\end{prop}
\begin{proof}
Let $M = (R \times S^1) \cup W_0 \cup...\cup W_r$ be the decomposition associated to the description of $M$ via its Seifert invariants $(g,b,\frac{\beta_1}{\alpha_1},...,\frac{\beta_r}{\alpha_r})$. We set $M' = M/\mathbb{Z}_n$ where $\mathbb{Z}_n \subset S^1$ denotes the $n$-th roots of unity and $M \stackrel{p} \longrightarrow M'$ is the quotient map. The quotient manifold has a natural $(S^1/\mathbb{Z}_n)$-action and thus $M'$ is again Seifert fibered and the map is fiber preserving. Furthermore, the decomposition of $M$ gives a decomposition $M'=(R \times S^1) \cup W'_0 \cup...\cup W'_r$ such that the restriction of the map $p$ to $R \times S^1$ is the product of the standard $n$-fold cover $S^1 \to S^1$ with the identity on $R$. Under the covering map the meridian class $m_i = -\alpha_i[R_i] + \beta_i[S^1]$ given by the $i$-th solid torus $W_i$ maps to $-\alpha_i[R_i] + n\beta_i[S^1]$, which must then be a multiple of the meridian class $m'_i$ along which $W'_i$ is attached:
$$m'_i = \frac{-\alpha_i}{\text{gcd}(n\beta_i,\alpha_i)}[R_i] + \frac{n\beta_i}{\text{gcd}(n\beta_i,\alpha_i)}[S^1].$$  
The divisibility of $p_*(m_i)$, which is $\text{gcd}(n\beta_i,\alpha_i) = \text{gcd}(n,\alpha_i) $ since $\alpha_i,\beta_i$ are coprime, then corresponds to the branching index of $p$ over the $i$-th exceptional fiber. 
\end{proof}

\subsection{Horizontal Foliations:} We next discuss horizontal foliations on Seifert manifolds referring to \cite{EHN} for further details. Here a foliation on a Seifert fibered space is called \emph{horizontal}, if it is everywhere transverse to the fibers of the Seifert fibration. A horizontal foliation $\mathcal{F}$ on a Seifert fibered space is equivalent to a representation $\widetilde{\rho}: \pi_1(M) \to \widetilde{\text{Diff}}_+(S^1)$, such that the homotopy class of the fiber is mapped to a generator of the centre of $\widetilde{\text{Diff}}_+(S^1)$, which acts on $\mathbb{R}$ as the group of $1$-periodic diffeomorphisms. One then has $M = (\widetilde{B} \times \mathbb{R} ) / \widetilde{\rho}$, where $\widetilde{B}$ denotes the universal cover of the quotient orbifold of $M$, and the horizontal foliation on the product descends to $\mathcal{F}$. The representation $\widetilde{\rho}$ then descends to a representation of the orbifold fundamental group of the base to the ordinary diffeomorphism group $\rho: \pi^{orb}_1(B) \to \text{Diff}_+(S^1)$. 

In all but a few cases a Seifert manifold admits a horizontal foliation if and only if it admits one with holonomy in $\textit{PSL}(2,\mathbb{R})$, in the sense that the image of the holonomy map in $\rho$ lies in $\textit{PSL}(2,\mathbb{R})$. Moreover, an examination of the proof of (\cite{EHN}, Theorem 3.2) and its analogue for $\textit{PSL}(2,\mathbb{R})$-foliations shows that it is always possible to ensure that the holonomy around some embedded curve in the base is hyperbolic provided that the base has positive genus. We note this in the following proposition.
\begin{prop}[Existence of horizontal foliations \cite{EHN}]\label{Seif_fol}
Let $M$ be a Seifert fibered space whose base has genus $g$, then $M$ admits a horizontal foliation if
$$2 - 2g - r \leq -b - r \leq 2g - 2.$$
In this case the horizontal foliation can be taken so as to have holonomy in $\textit{PSL}(2,\mathbb{R})$ and so that some embedded curve in the base whose holonomy is hyperbolic. If $g>0$ then the converse also holds.
\end{prop}
\noindent Thus in most cases the existence of a horizontal foliation on $M$ is the same as the existence of a flat connection on $M$ thought of as an orbifold $\textit{PSL}(2,\mathbb{R})$-bundle. In the case of genus zero, one has slightly more elaborate criteria for the existence of a $\textit{PSL}(2,\mathbb{R})$-foliation.
\begin{thm}[\cite{JN}, Theorem 1]\label{genus_zero_hor}
Let $M$ be a Seifert manifold with normalised invariants $(0,b,\frac{\beta_1}{\alpha_1},,...,\frac{\beta_r}{\alpha_r})$. Then $M$ admits a horizontal foliation with holonomy in $\textit{PSL}(2,\mathbb{R})$ if and only if one of the following holds:
\begin{itemize}
\item $2 - r \leq -b - r \leq - 2$
\item $b = -1$ and $\sum_{i=1}^r \frac{\beta_i}{\alpha_i} \leq 1$ or $b = 1-r$ and $\sum_{i=1}^r \frac{\beta_i}{\alpha_i} \geq r-1$.
\end{itemize}
\end{thm}

\section{Perturbing foliations}\label{perturb_fol}
In their book on confoliations, Thurston and Eliashberg showed how to perturb foliations to contact structures. In its most general form, their theorem shows that any $2$-dimensional foliation $\mathcal{F}$ that is not the product foliation on $S^2 \times S^1$ can be $C^0$-approximated by both positive and negative contact structures. Under additional assumptions on the holonomy of the foliation this perturbation can actually be realised as a \textbf{deformation}. That is, there is a smooth family $\xi_t$ of plane fields, such that $\xi_0$ is the tangent plane field of $\mathcal{F}$ and $\xi_t$ is contact for all $t > 0$. Moreover, if every closed leaf has linear holonomy or if the foliation is minimal with some holonomy, then $\mathcal{F}$ can be linearly deformed to a contact structure. Here a linear deformation is a $1$-parameter family of 1-forms $\alpha_t$ such that $Ker(\alpha_0) = T\mathcal{F}$ and
\[\frac{d}{dt}\alpha_t \wedge d \alpha_t \bigg|_{t=0} > 0.\]
This latter condition is then equivalent to the existence of a 1-form $\beta$ such that
\[\langle \alpha, \beta \rangle  =  \alpha \wedge d \beta + \beta \wedge d \alpha > 0.\]
Note further that
\[\langle f\alpha, f\beta \rangle = f^2\langle \alpha, \beta \rangle\]
so that the condition of being linearly deformable depends only on the foliation and not on the particular choice of defining $1$-form.
\begin{thm}[Eliashberg-Thurston \cite{ETh}]\label{spec_def}
Let $\mathcal{F}$ be a $C^2$-foliation that is not without holonomy.
\begin{enumerate}
\item If all closed leaves admit some curve with attracting holonomy. Then $T\mathcal{F}$ can be smoothly deformed to a positive resp.\ negative contact structure. 
\item If all closed leaves have linear holonomy, then this deformation can be chosen to be linear.
\end{enumerate}
\end{thm}
\begin{rem}
Foliations without holonomy are very special and can be $C^0$-approximated by surface fibrations over $S^1$. Thus the assumption that the foliation has some holonomy can be replaced by the topological assumption that the underlying manifold does not fiber over $S^1$. Examples of manifolds which cannot fiber are non-trivial $S^1$-bundles over surfaces of genus at least $2$, or more generally Seifert fibered spaces with non-trivial Euler class and hyperbolic quotient orbifolds, and rational homology spheres.
\end{rem}

In general it is not possible to deform families of foliations to contact structures in a smooth manner. However, if a smooth family of foliations $\mathcal{F}_{\tau}$ admits linear deformations for all $\tau$ in some compact parameter space $K$, then the fact that $\langle \alpha, \beta \rangle > 0$ is an open convex condition, means that one can use a partition of unity to smoothly deform the entire family. We note this in the following proposition, which will be mainly applied when the family has no closed leaves at all.
\begin{prop}[Deformation of families]\label{fam_def}
Let $\mathcal{F}_{\tau}$ be a smooth family of foliations that is parametrised by some compact space $K$ and suppose that each foliation in the family admits a linear deformation. Then $\mathcal{F}_{\tau}$ can be smoothly deformed to a family of positive resp.\ negative contact structures $\xi^{\pm}_{\tau}$.
\end{prop}
\noindent Another consequence of the convexity of the linear deformation condition is that the isotopy classes of any two linear deformations determining the same orientation is unique. This is an immediate consequence of Gray's Stability Theorem and we record this fact in the following:
\begin{prop}\label{lin_def}
Any two positive, resp.\ negative linear deformations of a foliation are isotopic.
\end{prop}


\section{Horizontal contact structures}\label{hor_contact}
Horizontal contact structures on Seifert manifolds, like horizontal foliations, may be thought of as connections with a certain curvature condition. As opposed to the flat case where the horizontal distribution is integrable, the distribution in question is contact if and only if the holonomy around the boundary of any embedded disc in the base is negative. To be precise this means that for the induced connection on the $\mathbb{R}$-bundle given by unwrapping the $S^1$-fibers over the disc the holonomy $h$ around the boundary of the disc satisfies $h(x)-x <0$. 
This then puts restrictions on the topology of Seifert manifolds that admit horizontal contact structures and one has the following necessary and sufficient conditions.
\begin{thm}[Honda \cite{Hon}, Lisca-Mati\'{c} \cite{LiM}]\label{hor_cond}
A Seifert manifold with normalised invariants $(g,b,\frac{\beta_1}{\alpha_1}, ... ,\frac{\beta_r}{\alpha_r})$ carries a (positive) contact structure transverse to the Seifert fibration if and only if one of the following holds:
\begin{itemize}
\item $-b-r \leq 2g - 2$
\item $g=0$, $r \leq 2$ and $-b-\sum \frac{\beta_i}{\alpha_i} < 0$
\item $g = 0$ and there are relatively prime integers $0 < a < m$ such that
$$\frac{\beta_1}{\alpha_1} > \frac{m-a}{m}, \frac{\beta_2}{\alpha_2} > \frac{a}{m} \text{ and } \frac{\beta_i}{\alpha_i} > \frac{m-1}{m} \text{, for } i \geq 3.$$ 
\end{itemize}
\end{thm}
\begin{rem}\label{realise}
The final condition is the realisability condition of \cite{EHN}, which is equivalent to the existence of a horizontal foliation by Naimi \cite{Nai}. For $g > 0$ the condition for the existence of a horizontal contact structure is the upper bound of the double sided inequalities that determine the existence of horizontal foliations (cf.\ Proposition \ref{Seif_fol}).
\end{rem}
\subsection{Classification results}
A given Seifert manifold can admit several isotopy classes of horizontal contact structures. An important invariant of contact structures on Seifert manifolds is the ``enroulement'' or twisting number as introduced by Giroux. For this recall that a Legendrian knot $K$ in a contact manifold $(M,\xi)$ inherits a canonical framing given by taking a vector field of unit normals along $K$ that are also tangent to $\xi$. After choosing a reference framing this gives an integer which is called \textbf{Thurston-Bennequin number}  $tb(K)$ of the Legendrian knot $K$.  
\begin{defn}[Giroux \cite{Gir}]\label{def_twisting}
Let $\xi$ be a contact structure on a Seifert fibered space. The \textbf{twisting number} $t(\xi)$ of $\xi$ is the maximal Thurston-Bennequin number of a knot that is smoothly isotopic to a regular fiber, where the Thurston-Bennequin number is measured relative to the canonical framing coming from the base. 
\end{defn}
Generalising results of Giroux \cite{Gir} from $S^1$-bundles to the case of general Seifert fibered spaces Massot \cite{Mas} has shown that a contact structure can be isotoped to a horizontal one if and only if it is universally tight and has negative twisting number:
\begin{thm}[\cite{Mas} Theorem A]\label{neg_twist}
Let $\xi$ be a contact structure on a Seifert fibered space. Then $\xi$ can be made horizontal via an isotopy if and only if it is universally tight and $t(\xi) <0$.
\end{thm}
An important step in the proof of Theorem \ref{neg_twist} is to show that any contact structure with $t(\xi) <0$ can be isotoped into a so called \textbf{normal form}. More precisely, given a Seifert fibered space described as $M =(R \times S^1) \cup W_0 \cup...\cup W_r$ we say that $\xi$ is in normal form if it is tangent to the $S^1$-fibers on $\widehat{M} = R \times S^1$. 
\begin{lem}[\cite{Mas} Proposition 5.5]\label{normal_exist}
Let $\xi$ be a contact structure on a Seifert fibered space with $t(\xi) < 0$. Then $\xi$ can be isotoped into normal form.
\end{lem}
In applications it will be important to use a slightly more precise version of this result, that follows from the way that Lemma \ref{normal_exist} is proved. 
\begin{lem}[\cite{Mas} pp.\ 1757--8]\label{normal_special}
Let $\xi$ be  a universally tight contact structure on a Seifert manifold $M$ and let $F_0$ be a regular fiber that is Legendrian and satisfies $tb(F_0) = t(\xi) < 0$. Then $\xi$ can be brought into normal form by an isotopy that fixes neighbourhoods of the exceptional fibers.
\end{lem}
\begin{proof}[Sketch of proof]
First assume that $F_0$ is a regular fiber over a base point $p_0$ in $R$ that is Legendrian and realises $t(\xi)$. Then by the Weinstein Neighbourhood Theorem for Legendrian knots we can assume after an isotopy with support near $F_0$ that the contact structure is vertical near $F_0$ and is given as the kernel of a $1$-form $\alpha_n = cos(n\theta)dx - sin(n \theta)dy$, where $\theta$ denotes the fiber coordinate and $(x,y)$ are coordinates on a neighbourhood of $p_0$. Here $n= -t(\xi)$ and it is essential that this number is positive so that the form $\alpha_n$ determines a positive contact structure. We then consider a bouquet of circles $B= \gamma_1 \vee \cdots \vee \gamma_k$ based at $p_0$ in $R$ onto which $R$ contracts. One then uses Giroux's Flexibility Theorem to make the contact structure vertical near $\gamma_i \times S^1$ by isotopies with support disjoint from a small neighbourhood of $F_0$. Again at this point it is essential that the twisting is negative in order to apply Giroux's Flexibility Theorem for surfaces (in this case annuli) with Legendrian boundary. These isotopies all have support in a small neighbourhood $N$ of $B \times S^1$. By stretching $N$ to fill out all of $\widehat{M}$ we obtain the desired isotopy, which in total has support disjoint from neighbourhoods of the exceptional fibers.
\end{proof}
Now any vertical contact structure $\xi$ on $\widehat{M} = R \times S^1$ is given as the pullback of the canonical contact structure $\xi_{can}$ under a fiberwise $n$-fold covering $\widehat{M} \rightarrow ST^*R$ that we denote $p_{\xi}$ (cf.\ \cite{Gir} Proposition 3.3). The map $p_{\xi}$ is defined by sending a point $x$ to the image of $\xi_x$ under the projection $\pi\colon R\times S^1 \to R$ considered as an element in $ST^*R$. Note that the number $n$ corresponds to $-t(\xi)$. 

Under the additional assumption that the contact structure is universally tight one can further restrict the possibilities for the contact structures on the solid torus neighbourhoods of the exceptional fibers. For in this case the contact structures on each of the solid tori $W_i$ is also universally tight. Thus according to classification results of Giroux (cf.\ \cite{Mas} Lemma 3.4) there are at most two possibilities for each $\xi|_{W_i}$ up to isotopy relative to $\partial W_i$. The two possibilities are distinguished by the fact that the are isotopic relative to the boundary to contact structures that are positively resp.\ negatively transverse to the fibers on $int(W_i)$. In fact, these choices must be made coherently due to the following:
\begin{lem}[\cite{Mas} Proposition 6.1]\label{unique_ex}
A contact structure $\xi$ on a Seifert fibered  space $M =(R \times S^1) \cup W_0 \cup...\cup W_r$ in normal form is universally tight if and only if the restrictions $\xi|_{W_i}$ are isotopic relative to $\partial W_i$ to ones that are \textbf{all} either positively or negatively transverse to the $S^1$-fibers on $int(W_i)$.

In particular, there are at most two universally tight extensions of $\xi|_{R \times S^1}$ which are isotopic as unoriented contact structures and they can both be made horizontal after a suitable isotopy.
\end{lem}
This lemma is extremely useful as it means that determining a universally tight contact structure with $t(\xi) <0$ up to isotopy (and orientation reversal of plane fields), reduces to determining it on $\widehat{M}$. In order to state Massot's classification of universally tight contact structures with negative twisting in terms of these normal forms it is convenient to describe the contact structure on $R \times S^1$ in a slightly different fashion.

For this one notes that the choice of section $\widehat{s}$ in $\widehat{M}$ used to compute the normalised Seifert invariants gives a section in $ST^*R$ via the covering map $p_{\xi}$. This section then gives an identification of $ST^*R$ and $\widehat{M}$ with $R \times S^1$ and with respect to these identifications the map $p_{\xi}$ is the product of the identity with the standard $n$-fold cover of $S^1$ up to fiberwise isotopy.

Moreover, under this identification the canonical contact structure is isotopic to the kernel of some $1$-form
\[\alpha_{\lambda} = \cos(\theta)\lambda + \sin(\theta)\lambda \circ J,\]
where $\lambda$ is a non-vanishing 1-form on $R$ and $J$ is an almost complex structure. The contact structure on $\widehat{M}$ is then given by the kernel of the $1$-form
\[\alpha_{\lambda,n} = \cos(n\theta)\lambda + \sin(n\theta)\lambda \circ J.\]
Note that the isotopy class of $\textrm{Ker}(\alpha_{\lambda,n})$ is independent of the choice of $J$. Furthermore, any homotopy of $\lambda$ through non-vanishing $1$-forms induces an isotopy of the contact structure $\textrm{Ker}(\alpha_{\lambda,n})$ and the homotopy class of $\lambda$ as a non-vanishing $1$-form is called the \textbf{$R$-class} of the normal form. The other important invariant of $\lambda$ is given by its indices on the boundary components of $R$ and the collection of these indices $(x_0,\cdots,x_r)$ ordered according to the tori $W_i$ is called the \textbf{multi-index} of the normal form. The various indices correspond to different fiberwise homotopy classes of maps
$$p_{\xi}\colon  \partial W_i \to S^1 \subseteq  \partial R,$$
which can in turn be identified with $\mathbb{Z}$ as a torsor.

We saw above that there are restrictions on the ways to extend a vertical contact structure on $R \times S^1$ to one that is universally tight. On the other hand there are simple arithmetic criteria to determine when this is possible. The following is a special case of (\cite{Mas}, Theorem B), which is stated only for Seifert fibered spaces whose invariants are in normal form, but its proof (cf.\ \cite{Mas} p.\ 1758) holds without this assumption.
\begin{lem}\label{arithemtic_crit}
Let $M$ be a Seifert fibered space with (not necessarily normalised) Seifert invariants $(g,b,\frac{\beta_1}{\alpha_1},...,\frac{\beta_r}{\alpha_r})$ and assume that
 $$b + \sum_{i=1}^r \Big\lceil \frac{\beta_i}{\alpha_i} \Big\rceil = 2-2g.$$
 Then there is a universally tight contact structure given by a normal form with multi-index $(b,\lceil \frac{\beta_1}{\alpha_1} \rceil,...,\lceil \frac{\beta_r}{\alpha_r} \rceil)$. In particular, $\xi$ can be assumed to be horizontal on the complement of $R \times S^1$.
\end{lem}
We now describe Massot's classification which divides into two cases. The first is the flexible case, where the specific normal form is not important, and the second is the rigid case where it contains essential information about the isotopy class of the contact structure. Note that in the latter case one requires the additional assumption that the base has genus $g > 0$.
\begin{thm}[Flexible Case, \cite{Mas} Theorem D]\label{normal_form_1}
Let $\xi,\xi'$ be universally tight contact structures on a Seifert fibered space with normalised Seifert invariants $(g,b,\frac{\beta_1}{\alpha_1},...,\frac{\beta_r}{\alpha_r})$. If $-b-r < 2g-2$ and $t(\xi) =t(\xi') =-1$, then $\xi$ and $\xi'$ are isotopic as unoriented contact structures, i.e.\ they are isotopic after possibly swapping the orientations of one of the plane fields.
\end{thm}
\begin{thm}[Rigid Case, \cite{Mas} Theorem E]\label{normal_form_2}
Let $\xi$ be a universally tight contact structure on a Seifert fibered space with normalised Seifert invariants $(g,b,\frac{\beta_1}{\alpha_1},...,\frac{\beta_r}{\alpha_r})$, where $g > 0$. Assume furthermore that either  $t(\xi) =-n < -1$ or $-b-r = 2g-2$ and $t(\xi) =-1$. Then the $R$-class of the normal form is unique. 

Moreover, the multi-index of this normal form is $(nb,\lceil \frac{n\beta_1}{\alpha_1}\rceil,...,\lceil \frac{n\beta_r}{\alpha_r}\rceil)$. 
\end{thm}
\noindent In the case where the base has genus $g = 0$, the second part of Theorem \ref{normal_form_2} still holds. The following is a special case of \cite{Mas} Proposition 8.2. Note that Massot uses the notation $e_0 = -b-r$ at this point.
\begin{thm}[\cite{Mas} Proposition 8.2]\label{normal_form_3}
Let $\xi$ be a universally tight contact structure on a Seifert fibered space with normalised Seifert invariants $(g,b,\frac{\beta_1}{\alpha_1},...,\frac{\beta_r}{\alpha_r})$, where $g = 0$. Assume furthermore that either $t(\xi) =-n < -1$ or $-b-r = 2g-2$ and $t(\xi) =-1$. Then the multi-index of any normal form is $(nb,\lceil \frac{n\beta_1}{\alpha_1}\rceil,...,\lceil \frac{n\beta_r}{\alpha_r}\rceil)$. In particular, there is at most one isotopy class of contact structures with $t(\xi) = -n$ up to orientation reversal of plane fields.
\end{thm}
\subsection{Realising contact structures as branched coverings} Contact structures on $S^1$-bundles over surfaces with twisting $-n$ can be realised as coverings of contact structures with twisting $-1$ (cf.\ \cite{Gir}). For more general Seifert fibered spaces a similar result holds, but in general one must allow branched coverings.
\begin{defn}[Contact branched cover]\label{cont_branched}
Let $(M',\xi')$ be a contact manifold with $\xi' = \text{Ker}\thinspace (\alpha')$ and let $M \stackrel{p} \longrightarrow  M'$ be a branched covering of $3$-manifolds such that the image of the branching locus $L \subset M$ under $p$ is a transverse link. Choose a $1$-form $\beta$ on $M$ with support in a neighbourhood of $L$ such that
$$p^*\alpha' \wedge d\beta> 0 \text{ and } \beta = 0  \text{ along } L.$$
The pull-back contact structure $p^* \xi'$ is then defined as the kernel of the following $1$-form for any $\epsilon > 0$ sufficiently small:
$$\alpha = p^*\alpha' + \epsilon \thinspace \beta.$$
\end{defn}
\begin{rem}
Since the conditions imposed on $\beta$ are convex, it follows that $p^* \xi'$ is well defined up to isotopy in view of Gray's Stability Theorem.
\end{rem}

\noindent Now if $M$ admits a contact structure $\xi$ with twisting number $-n$, then $\xi$ is in fact isotopic to the pullback of a contact structure $\xi'$ with twisting number $-1$ by an $n$-fold fiberwise branched cover.
\begin{prop}\label{cover_form}
Let $M$ be a Seifert manifold admitting a contact structure with twisting number $t(\xi) = -n < -1$, then there is a fiberwise branched covering $M \stackrel{p} \longrightarrow  M'$ and a contact structure $\xi'$ on $M'$ with twisting $-1$ such that $\xi$ is isotopic to $p^* \xi'$.
\end{prop}
\begin{proof}
Let $(g,b,\frac{\beta_1}{\alpha_1},...,\frac{\beta_r}{\alpha_r})$ be the normalised Seifert invariants of $M$ and let $M \stackrel{p} \longrightarrow M'$ be the $n$-fold fiberwise branched cover given by Proposition \ref{fibre_cov}. Since $t(\xi) = -n$ by assumption, the contact structure $\xi$ admits a normal form with associated $1$-form
$$\alpha_{\lambda,n} = \cos(n\theta) \thinspace \lambda + \sin(n\theta) \thinspace \lambda \circ J.$$
By Theorem \ref{normal_form_2} the indices of $\lambda$ are $(nb,\lceil \frac{n\beta_1}{\alpha_1}\rceil,...,\lceil \frac{n\beta_r}{\alpha_r}\rceil)$ and Poincar\'e-Hopf implies
\[ nb + \sum_{i=1}^r \Big\lceil \frac{n\beta_i}{\alpha_i} \Big\rceil = 2-2g.\]
It then follows by Lemma \ref{arithemtic_crit} that the Seifert manifold $M'$, which has Seifert invariants $(g,nb,\frac{n\beta_1}{\alpha_1},...,\frac{n\beta_r}{\alpha_r})$, admits a contact structure $\xi'$ with normal form given by 
$$\alpha_{\lambda} = \cos(\theta) \thinspace \lambda + \sin(\theta) \thinspace \lambda \circ J,$$
which is in particular transverse to the branching locus of the map $p$. The pullback of the contact structure $\xi'$ can then be perturbed in a $C^{\infty}$-small fashion near the branching locus, where the contact structure is transverse, to obtain $p^*\xi'$. Since there is a unique way to extend the pullback $p^*\xi'|_{\widehat{M}}$ to a contact structure on all of $M$, which can be made positively transverse by Lemma \ref{unique_ex} and $p^* \xi'$ is isotopic to a positively transverse contact structure, we conclude that $\xi$ is isotopic to $p^* \xi' $. 
\end{proof}
\noindent We next note that any two contact structures with twisting number $-1$ are necessarily contactomorphic modulo orientation reversal of plane fields. Note that this follows from Theorem \ref{normal_form_1} if $-b-r < 2g-2$. In the other case we have:
\begin{prop}\label{con_class}
Let $\xi,\xi'$ be universally tight contact structures on a Seifert fibered space with normalised Seifert invariants $(g,b,\frac{\beta_1}{\alpha_1},...,\frac{\beta_r}{\alpha_r})$ and assume that $-b-r = 2g-2$, the twisting numbers satisfy $t(\xi) = t(\xi') = -1$ and $g >0$. Then any two normal forms of $\xi$ and $\xi'$ are contactomorphic by a diffeomorphism with support disjoint from the exceptional fibers.
\end{prop}
\begin{proof}
Let $\alpha_{\lambda}$ and $\alpha_{\lambda'}$ be the 1-forms associated to the normal form of $\xi$ and $\xi'$ respectively on $\widehat{M} = R \times S^1 \subset M$. After possibly replacing $\lambda$ with $-\lambda$, we may assume that both $\xi$ and $\xi'$ are isotopic to \emph{positively} transverse contact structures. By Theorem \ref{normal_form_2} the indices of both $\lambda$ and $\lambda'$ must then agree on $\partial R$, since by assumption $-b-r = 2g-2$. This is equivalent to the restrictions of the maps
$$p_{\xi},p_{\xi'}: \widehat{M} \to ST^*R$$
being fiberwise isotopic on the boundary of $\widehat{M}$. Furthermore, since $t(\xi) = t(\xi') = -1$ the maps above are in fact diffeomorphisms so that after an initial isotopy we may assume that $p_{\xi}$ and $p_{\xi'}$ agree near $\partial \widehat{M}$. It follows that $p_{\xi} \circ p^{-1}_{\xi'}$ is a diffeomorphism of $\widehat{M}$ that extends to all of $M$ so that $\xi$ and $\xi'$ are contactomorphic. 
\end{proof}
\begin{rem}\label{rem_or_rev}
If $M$ admits an orientation preserving diffeomorphism that reverses the orientation on the fibers, then any oriented horizontal contact structure is contactomorphic to the contact structure given by reversing the orientation of the plane field. In this case the above proposition in fact holds for contactomorphism classes of \emph{oriented} contact structures. Examples of such manifolds are given by Brieskorn spheres $\Sigma(p,q,r) \subset \mathbb{C}^3$, in which case the conjugation map on $\mathbb{C}^3$ yields the desired self-diffeomorphism.
\end{rem}

\section{Deformations of taut foliations on Seifert manifolds}\label{Deform_taut}
In this section we consider the problem of determining which contact structures on Seifert manifolds are deformations of taut foliations. The obvious necessary condition for a contact structure to be a perturbation of a taut foliation is that it is universally tight. We will show that in most cases a universally tight contact structure $\xi$ with negative twisting on a Seifert fibered manifolds is indeed a deformation of a taut foliation. 


In fact by Proposition \ref{cover_form} it suffices to consider contact structures with twisting number $-1$, in which case it is fairly easy to construct the necessary foliations at least when the genus of the base is at least one. The genus zero case is more subtle as not every contact structure with negative twisting can be a perturbation of a taut foliation. We first note some preliminary lemmas.
\begin{lem}\label{cover_def}
Let $\xi = p^*\xi'$ be a contact structure on a Seifert fibered space which is the pullback of a contact structure $\xi'$ under a fibered branched cover $M \stackrel{p} \longrightarrow M'$. Assume that $\xi'$ is isotopic to a linear deformation of a taut foliation through an isotopy that is transverse to the branching locus of $p$. Then $\xi$ is also a linear deformation of a taut foliation.
%
\end{lem}
\begin{proof}
Let $\alpha'$ be a defining form for $\xi'$ and let $\alpha_t$ be a smooth family of non-vanishing $1$-forms so that $\text{Ker}(\alpha_0)$ is integrable and tangent to a taut foliation and $\alpha_t$ is contact for $t > 0$. After applying an initial isotopy, we may also assume that $\text{Ker}(\alpha_1) = \xi'$ and that the entire family is transverse to the branching locus $L$ of $p$. Then $p^*\alpha_t$ is a deformation of a taut foliation that is contact away from $L$ for $t>0$, where it is closed. We let $\beta$ be a $1$-form as in Definition \ref{cont_branched} such that $p^*\alpha_0 \wedge d\beta> 0$ and $\beta= 0$ along $L$ . Then $\widetilde{\alpha}_t = p^*\alpha_t + t \thinspace \epsilon \beta$ provides the desired linear deformation since
\[\frac{d}{dt}\widetilde{\alpha}_t \wedge d \widetilde{\alpha}_t \bigg|_{t=0} = p^*\left(\frac{d}{dt}\alpha_t \wedge d \alpha_t \right)\bigg|_{t=0}  + \epsilon \thinspace (p^*\alpha_0 \wedge d\beta +  \beta \wedge p^*d\alpha_0) > 0\]
for any $\epsilon > 0$ that is sufficiently small.
\end{proof}
\noindent We now come to the main result of this section.
\begin{thm}\label{negative_pert1}
Let $\xi$ be a universally tight contact structure with negative twisting number $t(\xi)=-n$ on a Seifert manifold and assume that the base orbifold has genus $g > 0$. Then $\xi$ is a deformation of a taut foliation. Moreover, if $n > 1$ or $2 - 2g \leq -b $ in the case that $n = 1$, then this foliation can be taken to be horizontal and the deformation linear.
\end{thm}
\begin{proof}
By Proposition \ref{cover_form} there is a fiberwise branched covering $M \stackrel{p} \longrightarrow M'$ and a horizontal contact structure $\xi'$ so that $\xi$ is isotopic to $p^*\xi'$ and $t(\xi') = - 1$. For convenience we assume that both $\xi$ and $\xi'$ are in normal form and that $p^*\xi' = \xi$. We let $(g,nb,\frac{n\beta_1}{\alpha_1},,...,\frac{n\beta_r}{\alpha_r})$ denote the unnormalised Seifert invariants of $M'$. By Theorem \ref{hor_cond} we have that $-b - r \leq 2g - 2$. If $n > 1$ then according to Theorem \ref{normal_form_2}, we also have
\begin{equation}\label{Massot_equation}
nb + \sum_{i=1}^r \Big\lceil \frac{n\beta_i}{\alpha_i} \Big\rceil = 2-2g
\end{equation}
so that the normalised invariants $(g,b',\frac{\beta'_1}{\alpha'_1},...,\frac{\beta'_r}{\alpha'_r})$ of $M'$ satisfy $-b'-r = 2g-2$. We consider several cases:

\medskip

\textbf{Case 1: $n > 1$:} 

\medskip

\noindent In this case $-b'-r = 2g-2$ on $M'$. Proposition \ref{Seif_fol} then gives a horizontal  $\textit{PSL}(2,\mathbb{R})$-foliation $\mathcal{F}$ on $M'$ with hyperbolic holonomy around some embedded curve $\gamma$. We may then apply part (2) of Theorem \ref{spec_def} to deform the foliation linearly to a horizontal contact structure $\xi_{hor}$ on $M'$. The characteristic foliation on the torus $T_{\gamma}$ corresponding to $\gamma$ is Morse-Smale and has two closed orbits each intersecting a fiber in a point. This is then stable under a suitably small linear deformation. In particular, $T_\gamma$ is convex in the sense of Giroux and has a dividing set with two components, each of which intersects the fiber once.  Thus applying Giroux's Flexibility Theorem there is an isotopy with support in a neighbourhood of $T_\gamma$ so that the torus becomes \emph{ruled} or in \emph{standard form} (cf.\ \cite{Hon0} Corollary 3.6), that is the $S^1$-fibers become Legendrian. These Legendrian fibers then have Thurston-Bennequin number $-1$ measured relative to the framing coming from the base. 

It follows that $t(\xi_{hor}) \geq -1$. The opposite inequality holds for all horizontal contact structures by Theorem \ref{neg_twist} and we conclude that $t(\xi_{hor}) = -1$. We may then isotope $\xi_{hor}$ into normal form through an isotopy that is fixed near the exceptional fibers of $M'$ by Lemma \ref{normal_special}. Since all contact structures with twisting number $-1$ are contactomorphic by Proposition \ref{con_class}, we may assume that the normal form of $\xi_{hor}$ agrees with that of $\xi'$ after applying a suitable diffeomorphism. Note that the diffeomorphism given in the proof of Proposition \ref{con_class} can be chosen with support disjoint from the singular locus. It follows that $\xi'$ is isotopic to a deformation of a taut foliation. Since this isotopy was chosen to satisfy the hypotheses of Lemma \ref{cover_def} the result follows in this case.

\medskip

\textbf{Case 2: $n = 1$ and $-b-r = 2g-2$:} 

\medskip

\noindent The argument from the previous case gives the result by taking $M = M'$.

\medskip

\textbf{Case 3a: $n = 1$ and  $b \le 2g-2 <-b-r$:} 

\medskip

\noindent Proposition \ref{Seif_fol}  gives a horizontal $\textit{PSL}(2,\mathbb{R})$-foliation $\mathcal{F}$ on $M$ with hyperbolic holonomy around some embedded curve $\gamma$, which can be linearly deformed to a horizontal contact structure $\xi_{hor}$.  This contact structure must again have $t(\xi_{hor}) = -1$ and thus (after being suitably oriented) is then isotopic to $\xi$ by Theorem \ref{normal_form_1}.

\medskip

\textbf{Case 3b: $n = 1$ and $b > 2g-2$:}

\medskip

\noindent In this case we must have $t(\xi)=-1$, since $g>0$ so that $b > 0$ and the equation (\ref{Massot_equation}) cannot have any solutions. Moreover, there is only one such contact structure on $M$ up to changing the orientation of the plane field by Theorem \ref{normal_form_1}. Thus it will suffice to show that some horizontal contact structure is a deformation of a taut foliation. To this end we let $\gamma$ be a homologically essential simple closed curve in the base orbifold $B$, which exists by our assumption that $g > 0$. We cut $M$ open along the torus $T_{\gamma}$ which is the preimage of $\gamma$ in $M$ and take any horizontal foliation on the complement of $T_{\gamma}$ whose holonomy is conjugate to a rotation on the two boundary components of $B \setminus \gamma$. We may assume that the rotation angles are distinct, unless $M=T^3$, in which case all tight contact structures are deformations of some product foliation (cf.\ \cite{ETh} Proposition 2.3.1). 

We then spiral this foliation along the torus $T_{\gamma}$ (cf.\ Section \ref{Fol_and_Cont}) to obtain a foliation with a unique torus leaf that is non-separating. For convenience we then insert a product foliation to obtain a foliation $\mathcal{F}_{\gamma}$ with a single stack of torus leaves all of which are non-separating, meaning that $\mathcal{F}_{\gamma}$ is in particular taut. If $\alpha_{-},\alpha_+$ denote closed forms defining the foliation on the boundary components of a tubular neighbourhood $T \times [-1,1]$ of $T_{\gamma}$, then $\mathcal{F}_{\gamma}$ is given as the kernel of the following $1$-form:
\[\alpha_0 = \rho(-z)\alpha_{-} + \rho(z)\alpha_+ + (1-\rho(|z|))dz,\]
where $z$ denotes the second coordinate in $T \times [-1,1]$ and $\rho$ is a non-decreasing function such that:
\[\rho(z) =\begin{cases}
 0,  & \text{if } z \leq \frac{1}{4}\\
  1, & \text{if  } z \geq \frac{3}{4}.
\end{cases}\]
We write
\begin{align*}
\alpha_{-} &= \cos(\theta_-) dx-  \sin(\theta_-)dy\\
\alpha_+ &= \cos(\theta_+) dx-  \sin(\theta_+)dy \ , \  -\pi<\theta_-, \theta_+ < 0.
\end{align*}
Here we identify the fiber direction of $M$ with the $y$-coordinate. After possibly swapping the orientation of the $z$ coordinate, and hence of $M$, we may assume that $\theta_- < \theta_+$. We can then deform $\mathcal{F}_{\gamma}$ to a confoliation that is contact near $T_{\gamma}$. On $T \times [-1,1]$ this is given by the following explicit deformation
\[\alpha_t = \alpha_0 + t\left(\cos(f(z))dx - \sin(f(z))dy\right)\]
for a non-decreasing function $f:\mathbb{R} \to [0, \frac{\pi}{2}]$ that is constant outside of $[-1,1]$, has positive derivative on $(-\frac{1}{4},\frac{1}{4})$ and satisfies
\[f(z) =\begin{cases}
 \theta_-,  & \text{if } z \leq -\frac{1}{4}\\
  \theta_+, & \text{if  } z \geq \frac{1}{4}.
\end{cases}\]
A simple calculation shows that $\xi_t = \text{Ker}(\alpha_t)$ is a positive confoliation that is contact on $T \times (-\frac{1}{4},\frac{1}{4})$ for $t > 0$. Furthermore, the form $\cos(f(z))dx - \sin(f(z))dy$ is positive on $S^1$-fibers so that $\xi_t$ is horizontal for $t \in (0,1]$. Since $-b < 2 - 2g$ by assumption, $M$ can admit a horizontal contact structure for one and only one orientation by Theorem \ref{hor_cond}, so the change of orientation made above does not affect anything. The resulting confoliation is then transitive and can thus be $C^{\infty}$-perturbed to a contact structure which is by construction horizontal. By (\cite{ETh} Proposition 2.8.3), this perturbation can then be altered to a deformation. 
\end{proof}
\noindent Under the assumption that the genus of the base $g=0$ we have the following special case of Theorem \ref{negative_pert1} that will be used below:
\begin{prop}\label{negative_pert2}
Let $\xi$ be a universally tight contact structure on a Seifert manifold $M$ with negative twisting number $-n$ and suppose that $\xi$ is isotopic to a vertical contact structure and the base orbifold is hyperbolic. Then $\xi$ is a linear deformation of a horizontal (and hence taut) foliation.
 \end{prop} 
 \begin{proof}
The fact that $\xi$ is isotopic to a vertical contact structure gives a natural $n$-fold covering
$$M \stackrel{p_{\xi}} \longrightarrow ST^*B,$$
where $ST^*B$ is the unit cotangent bundle of the base orbifold of $M$, which is in turn a compact quotient of $\textit{PSL}(2,\mathbb{R})$. The cotangent bundle $ST^*B$ carries a canonical contact structure $\xi_{can}$ which descends from a left-invariant one on $\textit{PSL}(2,\mathbb{R})$ and $p_{\xi}^* \thinspace \xi_{can} = \xi$. It is easy to see that this contact structure is a linear deformation of a taut foliation by considering the linking form on the Lie algebra of $\textit{PSL}(2,\mathbb{R})$ and thus the same holds on the quotient 
$ST^*B$ (cf.\ \cite{Bow}, Example 3.1). The proposition then follows by pulling back under $p_{\xi}$.
\end{proof}
Note that there can be no direct analogue of Theorem \ref{negative_pert1} in the genus zero case. For as a consequence of Theorem \ref{hor_cond}, there are Seifert manifolds that admit horizontal contact structures, but no taut foliations. A particularly interesting case is that of small Seifert fibered spaces which are those having $3$ exceptional fibers and base orbifold of genus $0$. In this case any universally tight contact structure must have negative twisting number, which is equivalent to being isotopic to a  horizontal contact structure. Furthermore, swapping the orientation of $M$ has the effect of changing $b$ to $-b+3$. Thus inspection of the criteria of Theorem \ref{hor_cond} shows that $M$ admits a horizontal contact structure in both orientations if and only if its invariants are realisable and hence this is equivalent to the existence of a horizontal foliation. For Seifert fibered spaces whose bases are of genus $g=0$ the existence of a taut foliation is equivalent to that of a horizontal foliation. We summarise in the following proposition, which is proved by Lisca and Stipsicz \cite{LiS} using Heegard-Floer homology, rather than using Theorem \ref{hor_cond} which can be proven by completely elementary methods (cf.\ \cite{Hon}).
\begin{prop}\label{existence_crit}
Let $M$ be a Seifert fibered space over a base of genus $g = 0$ with infinite fundamental group. Then the following are equivalent:
\begin{enumerate}
\item $M$ admits a universally tight contact structure with negative twisting number in both orientations
\item $M$ admits a horizontal contact structure in both orientations
\item $M$ admits a horizontal foliation
\item $M$ admits a taut foliation.
\end{enumerate}
If $M$ is small, the assumption on the twisting number can be removed in (1) and taut can be replaced by Reebless in (4).
\end{prop}
\noindent Note that it is not clear whether any given horizontal contact structure on a small Seifert fibered space is the deformation of a horizontal foliation in the case that both exist. However, in all likelihood this ought to be the case.

\section{Deformations of Reebless foliations on Seifert manifolds}\label{Deform_Reebless}
In this section we show that all universally tight contact structures $\xi$ with $t(\xi) \geq 0$ are deformations of Reebless foliations as soon as the obvious necessary conditions are satisfied. This follows from the existence of a normal form for such contact structures given in \cite{Mas2} which generalises Giroux's normal form for tight contact structures with non-negative twisting on $S^1$-bundles.

In the following a \emph{very small} Seifert fibered space is a Seifert fibered space that admits a Seifert fibering with at most 2 exceptional fibers and whose base orbifold has genus $g=0$. Note that a very small Seifert fibered space is either a Lens space (including $S^3$) or $S^1 \times S^2$. The Lens spaces do not admit Reebless foliations by Novikov's Theorem and the only Reebless foliation on $S^1 \times S^2$ is the product foliation, which cannot be perturbed to any contact structure. Thus it is natural to rule out such spaces when showing that certain contact structures are deformations of Reebless foliations. Furthermore, a folklore result of Eliashberg and Thurston \cite{ETh} states that a perturbation of a Reebless foliation is universally tight. Unfortunately, as pointed out by V.\ Colin \cite{Col} the proof given \emph{loc.\ cit.}\ contains a gap, and one only knows that there exists some perturbation that is universally tight. But in any case it is a reasonable assumption to make when considering which contact structures are deformations of Reebless foliations.

We first recall the existence of normal forms for universally tight contact structures with non-negative twisting.
\begin{thm}[\cite{Mas2} Theorem 3]\label{Massot_normal_gen}
Let $\xi$ be a universally tight contact structure on a Seifert manifold $M$ that is not very small and is not a $T^2$-bundle with finite order monodromy. If $t(\xi) \geq 0$, then $\xi$ is isotopic to a contact structure that is horizontal outside a (non-empty) collection of incompressible pre-Lagrangian vertical tori $T = \sqcup_{i=1}^N \  T_i$, on which the contact structure is itself vertical. 

Conversely, any two contact structures $\xi_0,\xi_1$ that are vertical on a fixed collection of pre-Lagrangian vertical tori $T$ and horizontal elsewhere are isotopic as unoriented contact structures.
\end{thm}

\noindent With the aid of the normal form described above it is now a simple matter to show the following.
\begin{thm}\label{Reebless_def}
Let $\xi$ be a universally tight contact structure on a Seifert fibered space $M$ with $t(\xi) \geq 0$ and assume that $M$ is not very small. Then $\xi$ is a deformation of a Reebless foliation.
\end{thm}
\begin{proof}

First assume that $M$ is not a torus bundle with finite order monodromy. Then by Theorem \ref{Massot_normal_gen} we may assume after a suitable isotopy that $\xi$ is horizontal away from a finite collection of tori $T = \sqcup_{i=1}^N \  T_i$ where $\xi$ is vertical. Now let $\mathcal{F}$ be any foliation which has the incompressible tori $T_i$ as closed leaves and is horizontal otherwise. We also require that the sign of the intersection of any fiber with $\mathcal{F}$ agrees with that of $\xi$ on $M \setminus T$ and that the holonomies around curves on either side of the images of each $T_i$ in the base orbifold $B$ of $M$ are rational rotations. Such foliations can easily be constructed by taking any horizontal foliation on the components of $M \setminus T$ that has the correct co-orientation and then spiralling into the torus leaves. Note that all torus leaves are incompressible so that $\mathcal{F}$ is Reebless.

We first thicken the foliation near each torus leaf by inserting a stack of torus leaves $T_i \times[-1,1]$. The resulting foliation can be deformed as in Case 3b in the proof of Theorem \ref{negative_pert1} to obtain a transitive confoliation $\xi'$ which is contact on $N_i = T_i \times[-\frac{1}{4},\frac{1}{4}]$. On $N_i$ this deformation is given by the explicit formula
\[\alpha_t = dz + t\left(\cos(f(z))dx - \sin(f(z))dy\right).\]
Here $z$ denotes the normal coordinate and $f$ is a monotone non-decreasing function such that 
\[f(z) =\begin{cases}
 \theta_-,  & \text{if } z \leq -\frac{1}{4}\\
  \theta_+, & \text{if  } z \geq \frac{1}{4},
\end{cases}\]
where $\theta_-,\theta_+$ are the (negative) angles of the foliations on the negative resp.\ positive side of $T_i$. Since the foliation is positively transverse on one side of the torus and negatively transverse on the other, we may furthermore assume that $\theta_-<\theta_+$ and that the interval $[\theta_-,\theta_+]$ contains either $0$ or $\pi$ but not \emph{both} and thus $\xi'$ becomes vertical precisely once on each $T_i \times[-1,1]$. The confoliation $\xi'$ can then be deformed to a contact structure which is horizontal on $M \setminus T$. By Theorem \ref{Massot_normal_gen} this contact structure is then isotopic to $\xi$. Finally the two step deformation of $\mathcal{F}$ can be achieved via a single deformation in view of (\cite{ETh} Proposition 2.8.3).

If $M$ is a torus bundle with finite order holonomy, then the universally tight contact structures are classified (see \cite{Gir2} \cite{Hon2}) and it is easy to see that they are all deformations of some $T^2$-fibration.
\end{proof}
\begin{rem}
Although the foliations in Theorem \ref{Reebless_def} are in general only Reebless, one can give sufficient conditions so that they are taut. For by replacing $dz$ with $-dz$ in the model used to define the foliation near the vertical tori $T_i$, one can arrange that the torus leaves have any given orientation. In particular, if one can orient the vertical tori $T_i$ of the normal form associated to $\xi$ in such a way that no collection of these tori is null-homologous, then the foliation $\mathcal{F}$ constructed above is taut in view of Theorem \ref{Goodman}.
\end{rem}

\section{Topology of the space of taut and horizontal foliations}\label{horizontal_and_taut}
The topology of the space of representations $\text{Rep}(\pi_1(\Sigma_g),\textit{PSL}(2,\mathbb{R}))$ for a closed surface group of genus $g \geq 2$ has been well studied and its connected components were determined by Goldman \cite{Gol}. Recall that for any topological group $G$ the representation space of a surface group $\pi_1(\Sigma_g)$ is
\[\{(\phi_1,\psi_1,...,\phi_g,\psi_g) \in G^{2g} \ | \ \prod_{i=1}^{g} [\phi_i,\psi_i] = 1\}.\]
In the case of $\textit{PSL}(2,\mathbb{R})$ the connected components of the representation space correspond to the preimages under the map given by the Euler class
$$ \text{Rep}(\pi_1(\Sigma_g),\textit{PSL}(2,\mathbb{R})) \stackrel{e} \longrightarrow [2-2g, 2g-2].$$
Moreover, the quotient of the connected component with maximal Euler class under the natural conjugation action is homeomorphic to Teichm\"uller space and is hence contractible. On the other hand the topology of the representation space $\text{Rep}(\pi_1(\Sigma_g),\text{Diff}_+(S^1))$ endowed with the natural $C^{\infty}$-topology, which can be interpreted as the space of foliated $S^1$-bundles after quotienting out by conjugation, is not as well understood. It would perhaps be natural to conjecture that the map induced by the inclusion 
$$G = \textit{PSL}(2,\mathbb{R}) \hookrightarrow \text{Diff}_+(S^1)$$
induces a weak homotopy equivalence on representation spaces or at least a bijection on path components. It is known that both representation spaces are path connected in the case of the maximal component (cf.\  \cite{Ghy}, \cite{Mat}). Indeed, results of Matsumoto and Ghys show that any maximal representation is smoothly conjugate to one that is Fuchsian. 


On the other hand, we will show that this is not the case for the space of representations with non-maximal Euler class. The basic observation is that the cyclic $n$-fold cover $G_n$ of $G=\textit{PSL}(2,\mathbb{R})$ also acts smoothly on the circle via $\mathbb{Z}_n$-equivariant diffeomorphisms so that there is a natural map
 $$\text{Rep}(\pi_1(\Sigma_g),G_n) \longrightarrow \text{Rep}(\pi_1(\Sigma_g),\text{Diff}_+(S^1)).$$
In general the images of these maps lie in different path components for different values of $n$ and fixed Euler class. 

We shall need some preliminaries concerning the relationship between horizontal foliations on $S^1$-bundles and their holonomy representations. For this we shall identify the universal cover $\widetilde{\text{Diff}}_+(S^1)$ of $\text{Diff}_+(S^1)$ with the group of diffeomorphisms of $\mathbb{R}$ that are periodic with respect to integer translations. We then consider
$$\widetilde{\text{Rep}}_e(\pi_1(\Sigma_g),\widetilde{\text{Diff}}_+(S^1)) = \{(\phi_1,\psi_1,...\, ,\phi_g,\psi_g) \in \widetilde{\text{Diff}}_+(S^1)^{2g} \ \  | \ \ \prod_{i=1}^{g} [\phi_i,\psi_i]  = \text{Tr}_{-e}\},$$
where $\text{Tr}_{-e}$ denotes a translation by an integer $-e$, where $e$ is the Euler number of the associated oriented $S^1$-bundle $M(e)$. Note that the space $\widetilde{\text{Rep}}_e(\pi_1(\Sigma_g),\widetilde{\text{Diff}}_+(S^1))$ can be identified with the space of representations $\text{Rep}(\pi_1(M(e)),\widetilde{\text{Diff}}_+(S^1))$ that send the fiber class $[S^1]$ to the translation $\text{Tr}_{1}$. The natural map
$$\widetilde{\text{Rep}}_e(\pi_1(\Sigma_g),\widetilde{\text{Diff}}_+(S^1)) \longrightarrow \text{Rep}_e(\pi_1(\Sigma_g), \text{Diff}_+(S^1))$$
is an abelian covering map, whose fiber can be identified with $H^1(\Sigma_g,\mathbb{Z})$ as a torsor. Finally we let $\text{Fol}_{hor}(M(e))$ denote the space of horizontal foliations on the bundle $M(e)$ with Euler class $e$, which then inherits a natural topology as a subspace of the space of sections of the oriented Grassmannian bundle, which can in turn be identified with the unit cotangent bundle $ST^*M(e)$. After the choice of a base point as well as a standard generators $a_i,b_i$ of $\pi_1(\Sigma_g)$ and a trivialisation of $M(e)$ over a neighbourhood of the bouquet of circles $a_1\vee b_1\vee \ldots \vee a_g\vee b_g$ one obtains a map
$$\text{Fol}_{hor}(M(e)) \longrightarrow \widetilde{\text{Rep}}_e(\pi_1(\Sigma_g),\widetilde{\text{Diff}}_+(S^1)).$$
This is obtained by considering the holonomy around a loop in $\Sigma_g$, which naturally gives a path in $\text{Diff}_+(S^1)$ with respect to the chosen trivialisation, and the homotopy class of this path then gives an element in $\widetilde{\text{Diff}}_+(S^1)$. Conversely given any element in $\widetilde{\text{Rep}}_e(\pi_1(\Sigma_g),\widetilde{\text{Diff}}_+(S^1))$ one can construct foliations with the given holonomy in a continuous manner.
\begin{lem}\label{holonomy_foliation}
The map $\text{Fol}_{hor}(M(e)) \longrightarrow \widetilde{\text{Rep}}_e(\pi_1(\Sigma_g),\widetilde{\text{Diff}}_+(S^1))$ admits a section. Moreover, any two foliations with the same associated holonomy representations are related by a bundle automorphism that is isotopic to the identity.
\end{lem}
\begin{proof}
We consider the standard cell structure
$$\Sigma_g = (a_1\vee b_1 \vee \ldots \vee  a_g \vee b_g) \cup D^2,$$
where $a_i,b_i$ denote representatives of the standard generators of $\pi_1(\Sigma_g)$. Now the bundle $M(e)$ is trivial over a neighbourhood $N$ of the $1$-skeleton $\Sigma_g^{(1)}$ and we choose a trivialisation $N \times S^1$. Note that $N$ can be described as the union of strips $\widehat{a}_i\times [-\epsilon,\epsilon], \widehat{b}_i \times [-\epsilon,\epsilon]$ attached to a small disc $D_0$ containing the base point $p_0$, where $\widehat{a}_i,\widehat{b}_i$ are closed intervals contained in $a_i,b_i$ respectively that are disjoint from the base point $p_0$. Given an element $\widetilde{\rho} \in \widetilde{\text{Rep}}(\pi_1(\Sigma_g),\widetilde{\text{Diff}}_+(S^1))$ we let $\{\phi^i_t\}_{t \in [0,1]}$ be the linear path joining $\widetilde{\rho}(a_i)$ to the identity. We define  $\{\psi^i_t\}_{t \in [0,1]}$ for $\widetilde{\rho}(b_i)$ in the same way. By reparametrising, we may assume that these paths are constant near the end points. We then define a foliation on $N \times S^1$ by pushing forward the product foliation using the paths $\phi^i_t,\psi^i_t$ on the parts of $M(e)$ lying over $\widehat{a}_i\times [-\epsilon,\epsilon], \widehat{b}_i \times [-\epsilon,\epsilon]$ and extending by the trivial product foliation on $D_0 \times S^1$.

We must then extend this over the $2$-cell $D^2$. Since the foliation is already defined over a neighbourhood of the $1$-skeleton, we only need to extend it over a slightly smaller disc $D'$ contained in the interior of $D^2$. To this end we choose a trivialisation of $M(e)$ over $D'$ and note that the foliation induces a loop $\gamma_t$ in $\text{Diff}_+(S^1)$ given as the holonomy around $\partial D'$ for some given basepoint $q$ near $p_0$. The loop $\gamma_t$ is contractible since $\widetilde{\rho}$ was an element in $ \widetilde{\text{Rep}}_e(\pi_1(\Sigma_g),\widetilde{\text{Diff}}_+(S^1))$ and thus lifts to a loop $\widetilde{\gamma}_t$ in $\widetilde{\text{Diff}}_+(S^1)$. This loop then extends to a map over $D'$ using linear paths to the identity. Furthermore, the composition 
$$D' \to \widetilde{\text{Diff}}_+(S^1) \to \text{Diff}_+(S^1)$$
determines a fiber preserving isotopy $\Phi_{\rho}$ of $D' \times S^1$ so that the pushforward of the product foliation $(\Phi_{\rho})_*\mathcal{F}_{prod}$ extends the foliation on $N \times S^1$. By construction we may assume that the foliation is the pullback of the suspension foliation determined by $\gamma_t$ on $\partial D' \times S^1$ via radial projection near the boundary. Thus we may assume that the resulting foliation $\mathcal{F}_{\rho}$ on $M(e)$ is smooth. The map $\rho \mapsto \mathcal{F}_{\rho}$ then defines the desired section, since the entire construction depends continuously on $\rho$.

Now suppose $\mathcal{F}_0, \mathcal{F}_1$ have the same holonomy representations. After an initial fiber preserving isotopy, we may assume that the foliations agree over a small disc $D_0$ near the base point $p_0$. Then since the holonomy representations agree, there are fiberwise automorphisms over $D_0 \cup \Sigma_g^{(1)}$ that are the identity near $D_0$ so that the induced $1$-dimensional foliations over $(a_1\vee b_1 \vee \ldots \vee  a_g \vee b_g) \setminus D_0$ agree. After a further fiberwise isotopy we may then assume that the foliations agree over a neighbourhood $N$ of the $1$-skeleton. Using the contractibility of $\widetilde{\text{Diff}}_+(S^1)$ one can then extend this isotopy over the $2$-cell $D^2$ relative to the boundary, which then gives the desired fiberwise isotopy.
\end{proof}
\begin{rem}\label{continuity_hol}
Note that Lemma \ref{holonomy_foliation} holds with respect to the $C^k$-topology for any $0\leq k \leq \infty$, where we consider the space of horizontal foliations as a subset of the space of sections of the unit cotangent bundle of the total space $M(e)$. This is because the construction uses only linear paths in $\widetilde{\text{Diff}}_+(S^1)$ and certain (fixed) cut-off functions, so that the tangent plane fields of the associated foliations depends continuously on the holonomy. In addition, the proof of Lemma \ref{holonomy_foliation} easily extends to show that the fiber of the holonomy map is in fact (weakly) contractible. We remark that this is a special feature of codimension-$1$ foliations and need not hold in higher codimension.
\end{rem}

\begin{rem}
We also note that the action of the full group of bundle automorphisms $\text{Aut}(M(e))$ on $\text{Fol}_{hor}(M(e))$ descends to an action on $\widetilde{\text{Rep}}_e(\pi_1(\Sigma_g),\widetilde{\text{Diff}}_+(S^1))$, which in turn corresponds to the action by the group of deck transformations $\text{Aut}(M(e))/\text{Aut}_0(M(e))$ of the covering map to $\text{Rep}_e(\pi_1(\Sigma_g), \text{Diff}_+(S^1))$ which can be identified with $H^1(\Sigma_g,\mathbb{Z})$.
\end{rem}
We have the following useful consequence of Lemma \ref{holonomy_foliation}.
\begin{cor}\label{continuity}
Let $\mathcal{F}$ be a foliation on $M(e)$ with holonomy $\widetilde{\rho}$ and let $\widetilde{\rho}_t $ be a $C^k$-continuous path of representations in $ \widetilde{\text{Rep}}_e(\pi_1(\Sigma_g),\widetilde{\text{Diff}}_+(S^1))$ with $\widetilde{\rho}_0=\widetilde{\rho}$. Then there is a family of foliations $\mathcal{F}_t$ whose tangent distributions depend continuously on $t$ in the $C^k$-norm for any $0\leq k \leq \infty$ and such that $\mathcal{F}_0 = \mathcal{F}$. 

Similarly, if $\widetilde{\rho}_n$ is a sequence converging in the $C^k$-sense to $\widetilde{\rho}$ then there is a sequence of horizontal foliations $\mathcal{F}_n$ so that the tangent distributions $T\mathcal{F}_n$ converge to $T\mathcal{F}$ in the $C^k$-sense.
\end{cor}
\begin{proof}
By Lemma \ref{holonomy_foliation} as well as Remark \ref{continuity_hol} there is a $C^k$-continuous path of horizontal foliations $\mathcal{F}_t$ on $M(e)$ whose holonomy representations are precisely  $\widetilde{\rho}_t $. Moreover, by the second part of Lemma \ref{holonomy_foliation} the original foliation $\mathcal{F}$ is fiberwise diffeomorphic to $\mathcal{F}_0$ by some diffeomorphism $\varphi$ so that the desired path is given by applying $\varphi$ to the path $\mathcal{F}_t$. The proof in the case of a sequence of representations follows \emph{mutatis mutandis}.
\end{proof}
\noindent Now that we have clarified the relationship between holonomy representations and horizontal foliations, we may now show that the space of representations $\text{Rep}(\pi_1(\Sigma_g), \text{Diff}_+(S^1))$ is in general not path connected.
\begin{thm}\label{Rep_disconnected}
Let $\#Comp(e)$ denote the number of path components of the space of representations $\text{Rep}_e(\pi_1(\Sigma_g), \text{Diff}_+(S^1))$ with fixed Euler class $e \neq 0$ such that $e$ divides $2g-2 > 0$ and write $2g-2 = n\thinspace e$. Then the following holds:
\[\#Comp(e) \geq n^{2g} +1.\]
\end{thm}
\begin{proof}
By (\cite{Gir} Th\'eor\`eme 3.1) there are horizontal contact structures $\xi_1$ and $\xi_n$ with twisting number $-1$ and $-n$ respectively on the $S^1$-bundle $M(e)$ with Euler class $e$. Furthermore, this contact structure $\xi_n$ can be made vertical, i.e.\ tangent to the $S^1$-fibers. Both contact structures $\xi_1$ and $\xi_n$ are linear deformations of a horizontal foliation by Theorem \ref{negative_pert1} and we let $\rho_1,\rho_n$ be the associated holonomy representations in $\text{Rep}(\pi_1(\Sigma_g), \text{Diff}_+(S^1))$. Assume that $\rho_{t}$ is a smooth family of representations joining $\rho_1$ to $\rho_{n}$. We first lift this path to $ \widetilde{\text{Rep}}_e(\pi_1(\Sigma_g),\widetilde{\text{Diff}}_+(S^1))$ and then let $\mathcal{F}_t$ denote the smooth family of foliations given by Lemma \ref{holonomy_foliation}. Note that the foliations we obtain in this way agree with the original foliations up to fiber preserving automorphism of the total space of the associated bundle $M(e)$. In particular, the twisting numbers of the contact structures obtained by linear perturbation of $\mathcal{F}_0$ and $\mathcal{F}_1$ agree with those of $\xi_1$ resp.\ $\xi_{n}$.

Since each foliation in the family $\mathcal{F}_t$ cannot have any closed leaves and $M(e)$ does not fiber over $S^1$, we may perturb the family linearly to a 1-parameter family of contact structures by Proposition \ref{fam_def}. It then follows from Proposition \ref{lin_def} that $\xi_1$ is isotopic to $\xi_{n}$, which is a contradiction. Thus both $\rho_1$ and $\rho_{n}$ lie in distinct components of $\text{Rep}(\pi_1(\Sigma_g), \text{Diff}_+(S^1))$.

The vertical contact structure $\xi_n$ determines a fiberwise $n$-fold cover of the unit cotangent bundle $ST^*\Sigma_g$. By (\cite{Gir} Lemme 3.9) the isotopy class of the associated $n$-fold covering is a deformation invariant of $\xi_n$. Since all foliations are only well defined up to fiber preserving automorphisms of $M(e)$, it follows that only the fiberwise isomorphism class of the covering is a deformation invariant of the associated foliation and hence of $\rho_n$. Moreover, isomorphism classes of fiberwise $n$-fold coverings are in one to one correspondence with elements in $H^1(\Sigma_g,\mathbb{Z}_n)$ and it follows that the numbers of path components of representations whose perturbations have twisting number $-n$ is at least $n^{2g}$. From this we conclude that
$$\#Comp(e) \geq n^{2g} +1. \qedhere$$
\end{proof}
\begin{rem}
For the sake of concreteness let us consider the representations $\rho_{2d},\rho_{st}$ given by a $(2d)$-fold fiberwise cover of the suspension of a Fuchsian representation in $\textit{PSL}(2,\mathbb{R})$ and the stabilisation of some Fuchsian representation respectively. More precisely, $\rho_{st}$ is the composition of a Fuchsian representation of a surface of Euler characteristic $\frac{1}{2d}(2-2g) = 2-2g'$ and a collapse map $\Sigma_g \to \Sigma_{g'}$ which collapses all but $g'$ handles to a point. These representations have the same Euler class but lie in different components of $\text{Rep}(\pi_1(\Sigma_g), \text{Diff}_+(S^1))$ since the corresponding contact structures have twisting $-n=-2d$ in the first case and $-1$ in the second. This then answers a question raised by Y.\ Mitsumatsu and E.\ Vogt in studying certain turbulisation constructions for 2-dimensional foliations on 4-manifolds.
\end{rem}
\begin{rem}
Let $G_n$ denote the $n$-fold covering of $\textit{PSL}(2,\mathbb{R})$. The proof of Theorem \ref{Rep_disconnected} shows that for fixed $n$ the components of the representation spaces $\text{Rep}(\pi_1(\Sigma_g), G_n)$ as computed by Goldman \cite{Gol} are distinguished by their contact perturbations. This applies both in the case that $n$ divides $2g-2$ and where it does not, although in the latter case there is only one component (\cite{Gol} Lemma 10.5) so the statement is uninteresting.
\end{rem}
Larcanch\'{e} \cite{Lar} also considered the problem of deforming taut foliations through certain restricted classes of foliations. She noted that on $T^2$-bundles over $S^1$ with Anosov monodromy of a certain kind, the stable and unstable foliations $\mathcal{F}_{s},\mathcal{F}_{u}$ cannot be deformed to one another through foliations without torus leaves. This uses Ghys and Sergiescu's classification results \cite{GhS} for foliations without closed leaves on such manifolds. However, $\mathcal{F}_{s}$ and $\mathcal{F}_{u}$ can be deformed to one another through taut foliations: one first spins both foliations along a fixed torus fiber to obtain foliations $\mathcal{F}'_{s},\mathcal{F}'_{u}$ with precisely one closed torus leaf $T$. On the complement of $T$ one has a foliation by cylinders on $T^2 \times (0,1)$ intersecting each fiber in a linear foliation. It is then easy to construct a deformation between $\mathcal{F}'_{s}$ and $\mathcal{F}'_{u}$ through foliations with one torus leaf which is homologically non-trivial. Thus we conclude that one can indeed deform $\mathcal{F}_{s}$ to $\mathcal{F}_{u}$ through taut foliations.

In view of this, it remains to find taut foliations that cannot be deformed to one another through taut foliations, although their tangent distributions are homotopic. We give two types of examples of this phenomenon: the first uses deformations and contact topology and the second uses the special structure of taut foliations on cotangent bundles.
\begin{thm}\label{Seifert_dis}
The space of taut foliations is in general not path connected on small Seifert fibered spaces.
\end{thm}
\begin{proof}
We let $M = -\Sigma(2,3,6k-1)$ be the link of the complex singularity $z_1^{2} +z_2^3 + z_3^{6k-1} = 0$ taken with the opposite orientation, which has Seifert invariants $(0,-2,\frac{1}{2},\frac{2}{3},\frac{5k-1}{6k-1})$. As noted in (\cite{Mas}, p.\ 1746) the Seifert manifold $M$ admits a vertical contact structure $\xi_{vert}$ that has twisting number $-(6k-7)$ if $k>1$, which is then a linear deformation of a taut foliation $\mathcal{F}$ by Proposition \ref{negative_pert2}. One further checks that the following holds
\begin{equation}\label{nec_cond}
-2n + \Big\lceil\frac{n}{2} \Big\rceil+\Big\lceil\frac{2n}{3} \Big\rceil+\Big\lceil\frac{n(5k-1)}{6k-1} \Big\rceil = 2
\end{equation}
if and only if $n = 6l-1$ and $1\leq l \leq k-1$. By Theorem \ref{normal_form_3} this is a necessary condition for the existence of a horizontal contact structure on $M$ with twisting number $-n$. Moreover, the quotient space of the $(6l-1)$-fold cover $M \stackrel{p} \longrightarrow M'_l$ given by Proposition \ref{fibre_cov} has normalised invariants $(0,-1,\frac{1}{2},\frac{1}{3},\frac{k-l}{6k-1})$ and thus admits a horizontal foliation $\mathcal{F}_{l}$ by Theorem \ref{genus_zero_hor}.
Since $\mathcal{F}_{l}$ cannot have any closed leaves and $M'_l$ does not fiber over $S^1$, the foliation $\mathcal{F}_{l}$ can be linearly deformed to a horizontal contact structure $\xi_l$. Now the corresponding necessary condition for the existence of a horizontal contact structure on $M'_l$ with twisting number $t(\xi_l)$ is obtained by substituting $n = -(6l -1)t(\xi_l)$ into equation (\ref{nec_cond}) and it follows that
$$-(6l-1)t(\xi_l) = 6l'-1, \text{ for some } l\leq l' \leq k-1.$$
We observes that the (negative) twisting number of a contact structure is sub-multiplicative under covering maps. Thus, if $6l-1$ is coprime to $6k-7$, then we deduce that 
$$-t(p^*\xi_l) \leq -(6l-1)t(\xi_l) < 6k-7$$
so that $\xi' = p^*\xi_l$ cannot be isotopic to $\xi_{vert}$. Note that $6l-1$ will be coprime to $6k-7$ for all values of $l$ such that $l >\frac{1}{6}( \sqrt{6k-7}+1)$ with at most one exception. 

Since $M$ is non-Haken all taut foliations are without closed leaves. Thus any path of taut foliations joining $\mathcal{F}$ to $\mathcal{F}' = p^* \mathcal{F}_l$ can be deformed to an isotopy between $\xi_{vert}$ and $\xi'$ by Propositions \ref{fam_def} and \ref{lin_def}, which yields a contradiction if $6l-1$ is coprime to $6k-7$.
\end{proof}
\begin{rem}
Since the arguments above only used the arithmetic properties of the Seifert invariants, they could also be applied to other small Seifert fibered manifolds. We also observe that the uniqueness results of Vogel \cite{Vog} give alternative proofs of Theorems \ref{Rep_disconnected} and \ref{Seifert_dis}. As his results only assume $C^0$-closeness they yield that the conclusions about path components also hold with respect to the weaker $C^0$-topology. All results also remain true for foliations that are only of class $C^2$ as this suffices for Theorem \ref{spec_def} and its various consequences.
\end{rem}
Since the foliations in Theorem \ref{Seifert_dis} are by construction horizontal, their tangent distributions are homotopic as oriented plane fields. Thus by \cite{Lar}, they are homotopic as foliations (cf.\ also \cite{Eyn}). The construction of such a homotopy of integrable plane fields introduces many Reeb components, so it is natural to ask whether this is necessary. The first example of this kind (\cite{Vog}, Example 9.5) was given by considering the oriented horizontal foliations $\mathcal{F}_{hor}$ and $\overline{\mathcal{F}}_{hor}$ on the Brieskorn sphere $ \Sigma(2,3,11)$ taken with the positive orientation. Using classification results of contact structures on $\Sigma(2,3,11)$, the horizontal foliation $\mathcal{F}_{hor}$ cannot be homotopic to $\overline{\mathcal{F}}_{hor}$ through oriented taut foliations, since the horizontal contact structures obtained as perturbations $\xi,\overline{\xi}$ are not isotopic. However, since all Brieskorn spheres admit orientation preserving diffeomorphisms that reverse the orientation of the Seifert fibration (Remark \ref{rem_or_rev}), these contact structures are in fact contactomorphic so that one cannot deduce disconnectedness for either diffeomorphism classes of taut foliations or for the space of unoriented taut foliations in this example.

On the other hand since the twisting number of a contact structure is a contactomorphism invariant and does not depend on the orientation of the plane field, the proof of Theorem \ref{Seifert_dis} shows that the space of taut foliations is also disconnected even if one only considers \emph{diffeomorphism classes} of unoriented foliations yielding the following
\begin{thm}\label{homotopy_Reeb}
There exist  an infinite family of manifolds $M_n$ each admitting a pair of taut foliations $\mathcal{F}_0,\mathcal{F}_1$ that are homotopic as oriented foliations but not as taut foliations. Furthermore, the same result holds true for unoriented foliations or if one considers diffeomorphism classes of foliations.
\end{thm}
\noindent Since the manifolds $\Sigma(2,3,6k-1)$ are non-Haken the notions of tautness and Reeblessness coincide and it follows that any homotopy between $\mathcal{F}_0$ and $\mathcal{F}_1$ in Theorem \ref{homotopy_Reeb} must have Reeb components. Moreover, these examples show that the space of taut foliations in a given homotopy class of plane fields can have more than one equivalence class up to deformation and diffeomorphism, but nonetheless we can only distinguish finitely many such equivalence classes.

It seems much more difficult to find examples where the number of equivalence classes is infinite. If instead one only considers deformation classes of taut foliations themselves without quotienting out by the action of the diffeomorphism group of the manifold as well, then it is possible to give examples where the number of components is infinite. This uses not only the special structure of foliations on the unit cotangent bundle of a hyperbolic surface but also the structure of a foliation near torus leaves. 

\subsection*{Torus leaves and Kopell's Lemma:}
The behaviour of a foliation near a torus leaf is well understood and is nicely described in Eynard-Bontemps's thesis \cite{Eyn}. The fundamental result that puts restrictions on the behaviour of a foliation of class at least $C^2$ near a torus leaf is the following lemma of Kopell.
\begin{lem}[Kopell]
Let $f$ and $g$ be commuting diffeomorphisms mapping $[0,1)$ into itself (not necessarily surjectively) that fix the origin and are of class $C^2$ and $C^1$ respectively. Assume that $f$ is contracting. Then either $g$ has no fixed point in $(0,1)$ or $g = Id$.
\end{lem}
\noindent One also knows that torus leaves occur in a finite number of stacks in the following sense (cf.\ Eynard-Bontemps \cite{Eyn}, Thurston \cite{Thu}).
\begin{lem}\label{torus_stacks}
Let $\mathcal{F}$ be a $2$-dimensional foliation on a closed $3$-manifold $M$. Then either $\mathcal{F}$ is a foliation by tori on a $T^2$-bundle over $S^1$ or there is a finite collection $N_i = T_i \times[0,c_i]$ of foliated $I$-bundles over $T^2$ in $M$ with $c_i \geq 0$ so that $T_i \times\{0\}$ and $T_i \times\{c_i\}$ are torus leaves and such that $M \setminus \cup \thinspace N_i$ contains no further torus leaves.
\end{lem}
Now for any stack $N_i$ as in Lemma \ref{torus_stacks}, one has an induced holonomy homomorphism $\rho_{\mathcal{F}}$ defined near each end of $N_i$. Since there are no torus leaves outside of the stacks $N_i$, it follows with the help of Kopell's Lemma that the holonomy $f$ around some loop near say the upper end of $N_i$ must be contracting on the outside of $N_i$. Then by a result of Szekeres \cite{Sze} one knows that $f$ is the time $1$ map of a flow generated by a $C^1$-vector field $u(z)\partial_z$ that is smooth away from the fixed point $c_i$ (cf.\ \cite{Eyn} Th\'eor\`eme 1.1). Once again by Kopell's Lemma it follows that the entire image of $\rho_{\mathcal{F}}$ is generated by elements contained in the flow generated by $u(z)\partial_z$. One can then conjugate the foliation to one whose characteristic foliation is linear on tori near the ends of a stack of torus leaves (cf.\ \cite{Eyn}, Lemme 5.21).
\begin{lem}\label{szek_lem}
Let $\mathcal{F}$ be a smooth foliation on $T^2 \times [0,\epsilon]$ having only $L = T^2 \times \{0\}$ as a closed leaf and let $(x,y,z)$ denote standard coordinates on $T^2 \times [0,\epsilon]$. Then there is a $C^1$-diffeomorphism $\phi$ mapping $T^2 \times [0,\epsilon]$ into itself that fixes $L$ and is smooth on $T^2 \times (0,\epsilon]$ such that the image of $\mathcal{F}$ under $\phi$ is defined by the kernel of the $1$-form
$$dz + u(z)(a \thinspace dx + b \thinspace dy)$$
for some function $u(z) \geq 0$ that is positive away from $L$.
\end{lem}
\noindent For a stack of torus leaves we let $\lambda_{\pm} = (a_{\pm} ,b_{\pm} )$ be the asymptotic slope near the positive resp.\ negative end of a stack normalised so that $||\lambda_{\pm}|| = 1$. Note that near the negative end we take coordinates of the form $T^2 \times [-\epsilon,0]$. Now if the slopes $\lambda_-$ and $\lambda_+$ do not coincide then the stack of leaves is \textbf{stable} in the sense that any foliation in a $C^0$-neighbourhood of $\mathcal{F}$ has a closed torus leaf in a neighbourhood of $N_i$. If a stack of tori has arbitrarily small perturbations that are without closed leaves then the stack is called \textbf{unstable}. 

The final ingredient we shall need is Thurston's straightening procedure for foliations on $S^1$-bundles (see also \cite{Cal} p.\ 178 ff.).

\begin{thm}[Thurston \cite{Thu}]\label{Thurston_straight}
Let $\mathcal{F}$ be a foliation on an $S^1$-bundle over a surface $\Sigma_g$ of genus $g \geq 2$ without closed leaves. Then $\mathcal{F}$ is isotopic to a horizontal foliation. Furthermore, if $\mathcal{F}$ is already horizontal on a vertical torus $T$, then this isotopy can be made relative to $T$.
\end{thm}
\begin{rem}\label{add_hyp}
We note that if a vertical torus $T$ is merely transverse to the foliation then we can assume after a suitable isotopy that the foliation is in fact horizontal on $T$. For since any foliation without closed leaves is isotopic to a horizontal one, it follows that $[\widetilde{L}]\cdot [S^1] =1$ for any leaf $\widetilde{L}$ of the foliation given by pulling back under covering induced by the universal cover of the base, when $\mathcal{F}$ is suitably oriented. This means that after a suitable isotopy all closed leaves of the induced non-singular foliation $\mathcal{F}|_{T}$ intersect each fiber positively (so that this intersection is non-empty). In particular, $\mathcal{F}|_{T}$ has no $2$-dimensional Reeb components and no closed orbit can be isotopic to a fiber. Consequently one can apply an isotopy of the foliation with support near $T$ so that $\mathcal{F}$ is horizontal on $T$.
\end{rem}
\noindent We are now ready to prove that the space of taut foliations on $ST^*\Sigma_g$ has infinitely many components. Before giving the proof we clarify our orientation conventions concerning the Euler class of the (unit) cotangent bundle of a surface. Any real oriented vector bundle of rank-$2$ determines a unique complex line bundle. For the tangent bundle of an oriented surface this is equivalent to the choice of an almost complex structure that is compatible with the orientation and the first Chern class of this complex line bundle $T_{\mathbb{C}} \Sigma_g$ agrees (more or less by definition) with the Euler class of the tangent bundle as an oriented rank-$2$ bundle. Thus it is natural to identify $T^*\Sigma_g$ with the dual of the complex bundle so that the Euler number satisfies
$$e = \langle e(ST^*\Sigma_g),[\Sigma_g]\rangle=\langle c_1(T^*_{\mathbb{C}} \Sigma_g),[\Sigma_g]\rangle=\langle-c_1(T_{\mathbb{C}}\Sigma_g),[\Sigma_g]\rangle =  2g-2.$$
\begin{thm}\label{taut_inf}
The space of taut foliations on $ST^*\Sigma_g$ has at least $\mathbb{Z}^{2g}$ components if $g\geq 2$, which are all pairwise homotopic as foliations.
\end{thm}
\begin{proof}
By Giroux \cite{Gir}, Honda \cite{Hon2} all horizontal contact structures on $ST^*\Sigma_g$ are contactomorphic and can be made vertical. Furthermore, their isotopy classes are parametrised by $H^1(\Sigma_g,\mathbb{Z}) \cong \mathbb{Z}^{2g}$, where $H^1(\Sigma_g,\mathbb{Z})$ is identified with the set $[\Sigma_g, S^1]$ of homotopy class of maps $\Sigma_g \to S^1$ which in turn acts on isotopy classes of contact structures via gauge transformations. Choose $\xi_{vert},\xi'_{vert}$ \emph{non-isotopic} vertical contact structures. Such contact structures are linear deformations of foliations $\mathcal{F},\mathcal{F}'$ by Theorem \ref{negative_pert1}. In fact, identifying a vertical contact structure with the canonical contact structure on $ST^*\Sigma_g$ it is easy to see that they are linear deformations of foliations that are descended from left invariant foliations on $\textit{PSL}(2,\mathbb{R})$ (cf.\ \cite{Bow}).

Now suppose $\mathcal{F}_t$ is a smooth family of taut foliations joining $\mathcal{F}$ and $\mathcal{F}'$. Note that the condition of having no closed leaves is an open condition: for if $\mathcal{F}_{s}$ has no closed leaves then by Theorem \ref{Thurston_straight} it is isotopic to a horizontal foliation, and since the Euler number $e(ST^*\Sigma_g) =2g-2 \ne 0$ by assumption, any foliation sufficiently close to $\mathcal{F}_{s}$ is also without closed leaves. Thus the set of $t$ for which $\mathcal{F}_t$ is without closed leaves is open and non-empty since both foliations $\mathcal{F}$ and $\mathcal{F}'$ are without closed leaves, or equivalently the set of values of $t$ for which $\mathcal{F}_t$ has a closed leaf is closed (and possibly empty). In view of this we let $0<t_0<1$ be the smallest value such that $\mathcal{F}_{t_0}$ has closed leaves, which must exist. Otherwise we could linearly perturb the deformation by Proposition \ref{fam_def} to obtain a contradiction, since by assumption the contact deformations $\xi_{vert}$ and $\xi'_{vert}$ of $\mathcal{F}, \mathcal{F}'$ respectively are not isotopic. Note that all the closed leaves of $\mathcal{F}_{t_0}$ are incompressible tori. There is then a finite collection of embeddings $N_i = T_i \times[0,c_i]$ so that the foliation contains no closed leaves outside the union of the $N_i$ and both $T_i \times\{0\}$ and $T_i \times\{c_i\}$ are closed leaves by Lemma \ref{torus_stacks}. After an isotopy we may assume that the $T_i$ are vertical tori (i.e.\ they are tangent to the $S^1$-fibers) and we let $\gamma_i$ denote their image curves in $\Sigma_g$. 

All stacks of torus leaves must be unstable by our assumption on $t_0$, otherwise all foliations $\mathcal{F}_t$ with $t$ sufficiently close to $t_0$ would again have closed (torus) leaves. In particular, the asymptotic slopes of tori near both ends of $N_i$ must agree. We let $N'_i$ be a slight thickening of $N_i$. We then cut out $N'_i$ and reglue along the boundary to obtain a smooth foliation $\mathcal{F}''$. We let $T_i$ be the vertical torus corresponding to the stack $N_i$ and note that the foliation $\mathcal{F}''$ restricts to a linear foliation on each $T_i$. 

We first claim that $\mathcal{F}''$ cannot have any closed leaves. Since each closed leaf in the original foliation was contained in one of the sets $N_i$, any closed leaf of $\mathcal{F}''$ must intersect at least one of the tori $T_i$. We suppose that $L$ is a closed leaf of $\mathcal{F}''$ and will derive a contradiction. The intersection of $L$ with each torus $T_{i}$ (when non-empty) is a collection of parallel circles that are isotopic to the fibers of $M = ST^*\Sigma_g$. Let $A$ be the closure of one of the annular components of $L \setminus ( \cup_i T_i) $ in the completion $\widehat{M}$ of $M \setminus ( \cup_i T_i) $ that is obtained by adding a positive and a negative boundary torus $T^{\pm}_i$ for each $T_i$. The foliation $\mathcal{F}''$ on $M \setminus ( \cup_i T_i) $ extends  to a foliation $\widehat{\mathcal{F}}$ on $\widehat{M}$ in a natural way. Since the characteristic foliation on each $T^{\pm}_i$ is linear, the holonomy of the foliation $\widehat{\mathcal{F}}$ around a boundary component of $\partial A$ is trivial. Thus a neighbourhood of $A$ in $\widehat{M}$ is also foliated by annuli. Using the compactness of the space of compact leaves it follows that the maximal $1$-dimensional family $\{A_t\}_{t \in [0,1]}$ of annuli containing $A$ then contains the boundary tori which $A$ intersects. Hence the family $\{A_t\}_{t \in [0,1]}$ fibers an annulus bundle over $S^1$ inside $\widehat{M}$ and this is the case for each component of $L \setminus ( \cup_i T_i) $. After reidentifying the boundary tori $T^{+}_i$ and $T^{-}_i$ this would give a description of $M$ as a union of annulus bundles over $S^1$ glued along boundary tori. But then $M =ST^*\Sigma_g$ would be fibered by tori, which is obviously a contradiction.

Thus as $\mathcal{F}''$ cannot have any closed leaves, we may apply Thurston's straightening procedure relative to each $T_i$ to obtain a horizontal foliation such that the intersection of $\mathcal{F}''$ with $T_i$ is linear. This is equivalent to the fact that the holonomy around $\gamma_i$ is conjugate to a rotation contradicting (\cite{Mat}, Theorem 2.2) and we conclude that no foliation in the family can have closed leaves. It follows that the family cannot exist and that $\mathcal{F}$ and $\mathcal{F}'$ cannot be deformed to one another through taut foliations. Since there are $\mathbb{Z}^{2g}$ different isotopy classes of contact structures, there are at least this many deformation classes of taut foliations. Finally since the foliations we are considering are by construction horizontal, their tangent distributions are homotopic as plane fields and thus by Larcanch\'e \cite{Lar} they are homotopic as integrable plane fields.
\end{proof}
In fact, the proof of Theorem \ref{taut_inf} shows that if a family of taut foliations $\mathcal{F}_t$ on $ST^*\Sigma_g$ contains a foliation which does not have closed leaves, then the same is true for the entire family. This observation also applies to families of Reebless foliations. Furthermore, since a foliation on $ST^*\Sigma_g$ without closed leaves is isotopic to the suspension foliation given by a Fuchsian representation in view of \cite{Ghy}, we deduce the following corollary.
\begin{cor}\label{Reebless_comp}
Let $\mathcal{F}_{hor}$ be a horizontal foliation on the unit cotangent bundle of a hyperbolic surface $ST^*\Sigma_g$. Then any foliation in the path component of $\mathcal{F}_{hor}$ in the space of Reebless foliations is isotopic to a foliation given by the suspension of a Fuchsian representation.
\end{cor}
It is easy to construct taut foliations $\mathcal{F}_{T}$ with a single vertical torus leaf on any $S^1$-bundle as long as the base has positive genus and we may assume that the tangent distribution of such a foliation is homotopic to a horizontal distribution.  In view of Corollary \ref{Reebless_comp} there can be no Reebless deformation between $\mathcal{F}_{T}$ and any horizontal foliation, even if one allows diffeomorphisms of either foliation. The same applies to any pair of diffeomorphic horizontal foliations whose contact perturbations are not isotopic.

The arguments above also apply to deformations of $C^2$-foliations that are only continuous with respect to the $C^0$-topology. Note, however, that the foliations $\mathcal{F}_T$ and $\mathcal{F}_{hor}$ can in fact be deformed to one another through taut foliations that are only of class $C^0$. This follows by first spinning the horizontal foliation $\mathcal{F}_{hor}$ along the vertical torus $T$, which can be done in a $C^0$-manner.  The remainder of the foliation is determined by a representation of a free group to $\widetilde{\text{Diff}}_+(S^1)$. Joining any two such representations arbitrarily and spiralling into $T$ then gives the desired deformation. We note this in the following proposition:
\begin{prop}\label{C_0_difference}
There exist pairs of smooth taut foliations which can be deformed to one another through taut $C^{0}$-foliations, but not through taut $C^{\infty}$-foliations.
\end{prop}

\section{Anosov foliations}\label{Anosov_fol}
In this section we give an alternative approach to the results obtained above that uses the classification of Anosov foliations of Ghys \cite{Ghy} and results of Matsumoto \cite{Mat0}, \cite{Mat}. We will call a representation \emph{Anosov} if its associated suspension foliation is diffeomorphic to the weak stable foliation of an Anosov flow. Recall that a flow $\Phi^t_X$ generated by a vector field $X$ on a closed 3-manifold $M$ is Anosov if the tangent bundle splits as a sum of line bundles that are invariant under the flow
$$TM = E^{u} \oplus E^{s} \oplus X$$
such that for some choice of metric and $C, \lambda >0$ and for any $t >0$
$$||(\Phi^t_X)_*(v_{u})|| \geq C^{-1} e^{\lambda t} ||v_{u}|| \text{ and }  ||(\Phi^t_X)_*(v_{s})|| \leq C e^{-\lambda t} ||v_{s}||,$$
where $v_{u}\in E^{u}, v_{s} \in E^{s}$. The line fields $E^{u}, E^{s}$ are called the strong unstable resp.\ stable foliations of the flow and the foliations $\mathcal{F}_u, \mathcal{F}_s$ tangent to the integrable plane fields 
$$ E^{u} \oplus X \ , \  E^{s} \oplus X$$
are called the weak unstable resp.\ stable foliations of the flow. In general, the weak stable resp.\ unstable foliations need not be smooth and this puts strong restrictions on the possible flows as we will see below. However, as we are interested in smooth foliations we will always assume that both the weak stable and unstable foliations are \emph{smooth}.

An important property of Anosov flows and foliations is their structural stability. $C^1$-stability of the Anosov condition goes back to Anosov's original paper \cite{Ano}. Moreover, the dynamics of an Anosov representation in terms of its \emph{translation numbers} also turn out to be $C^0$-stable. 
\begin{defn}[Translation number]
Let $\widetilde{\phi} \in \widetilde{\text{Homeo}}_+(S^1)$ be considered as a $1$-periodic diffeomorphism of $\mathbb{R}$. The translation number is defined as
$$\textrm{tr}(\widetilde{\phi}) = \lim_{n \to \infty}\frac{\widetilde{\phi}^n(x)}{n}.$$
\end{defn}
\noindent If $\phi \in \text{Homeo}_+(S^1)$ then the \emph{rotation number} $\textrm{rot}(\phi)$ is defined as the image of $\textrm{tr}(\widetilde{\phi})$ in $S^1 = \mathbb{R}/\mathbb{Z}$ for any lift of $\phi$ to $\widetilde{\text{Homeo}}_+(S^1)$. 

The prime examples of Anosov foliations come from the weak stable foliations of geodesic flows on hyperbolic surfaces and suspensions of linear Anosov maps on $T^2$. If one assumes that the weak stable foliation of an Anosov flow is required to be sufficiently regular, then Ghys \cite{Ghy} has shown that these are the only possibilities up to taking finite covers. This result is usually only stated in the case of the unit tangent bundle of a hyperbolic surface (\cite{Ghy} Theorem 5.3). A more general result that classifies smooth Anosov foliations on all $S^1$-bundles can easily be gleaned from Ghys' arguments, but as there is no clear statement of this in \cite{Ghy} we include a brief account of the main steps to obtain this classification (cf.\ Theorem \ref{Ghys_top} below).

If a smooth Anosov flow has sufficiently smooth weak stable foliation (in fact $C^{1,1}$ would suffice), then it admits a transverse projective structure.
\begin{thm}[\cite{Ghy} Theorem 4.1]\label{Ghys_proj}
Let $X$ generate a smooth Anosov flow on a closed $3$-manifold whose weak stable foliation $\mathcal{F}_s$ is smooth. Then $\mathcal{F}_s$ admits a transverse projective structure. 
\end{thm}
\noindent Using this projective structure one then gets a topological classification of Anosov flows whose weak stable foliations are smooth by results of Barbot.
\begin{thm}[\cite{Ghy} Theorem 4.7]\label{Ghys_proj_2}
Let $X$ generate a smooth Anosov flow on a closed $3$-manifold whose weak stable foliation $\mathcal{F}_s$ is smooth. Then $\mathcal{F}_s$ is topologically conjugate to the weak stable foliation of the suspension of a linear Anosov diffeomorphism on $T^2$ or to a homogeneous Anosov flow on $\widetilde{PSL(2,\mathbb{R})}/\widetilde{\Gamma}$ for some co-compact lattice $\widetilde{\Gamma}$ lying in the universal cover of $\textit{PSL}(2,\mathbb{R})$.
\end{thm}
As a consequence of Theorem \ref{Ghys_proj_2} we obtain that if the foliation in question is a suspension foliation on an $S^1$-bundle, then the holonomy representation is topologically conjugate to the inclusion of a lattice  $$\pi_1(\Sigma_g)\cong \Gamma \hookrightarrow G_n,$$
where $G_n$ denotes some $n$-fold cover of $\textit{PSL}(2,\mathbb{R})$.

In order to obtain a smooth classification of Anosov flows/foliations Ghys then observes that the techniques used to prove Theorem \ref{Ghys_proj} apply more generally to flows that are only \emph{topologically} conjugate to an Anosov flow whose weak stable foliation is smooth (cf.\ \cite{Ghy} Section 5). In this way there is a version of Theorem \ref{Ghys_proj} which holds under the weaker assumption that a suspension foliation is only \emph{topologically conjugate} to an Anosov foliation and this yields the following smooth classification result.
\begin{thm}[Smooth Anosov Foliations on $S^1$-bundles \cite{Ghy}]\label{Ghys_top}
Let $\mathcal{F}$ be a suspension foliation on an $S^1$-bundle whose holonomy map is topologically conjugate to one that is Anosov. Then the holonomy map of $\mathcal{F}$ is smoothly conjugate to a suspension of a co-compact lattice in a finite cover of $\textit{PSL}(2,\mathbb{R})$.

In particular, any Anosov foliation on an $S^1$-bundle is smoothly conjugate to the suspension of a co-compact lattice in some finite cover of $\textit{PSL}(2,\mathbb{R})$.
\end{thm}
\begin{proof}[Sketch of proof] 
One first alters the smooth structure on the underlying $S^1$-bundle so that the foliation $\mathcal{F}$ becomes the weak stable foliation of a smooth flow that is topologically conjugate to one that is Anosov and has smooth weak stable foliation (cf.\ \cite{Ghy} pp.\ 177--8). The resulting weak stable foliation is then topologically conjugate to $\mathcal{F}$ and we denote it by $\mathcal{F}'$. 

The arguments used to obtain a transverse projective structure for $\mathcal{F}'$ (Markov Partitions, etc.), then also give a transverse projective structure on $\mathcal{F}$ (cf.\ \cite{Ghy} pp.\ 179--181). It follows from the classification of projective structures on $S^1$ as described in the proof of (\cite{Ghy} Lemma 5.1) that the holonomy group of $\mathcal{F}$ is then \emph{smoothly conjugate} to a subgroup of $G_n$ (which \emph{a priori} may not be a lattice).  By Theorem \ref{Ghys_proj_2} the holonomy group is also topologically conjugate to a lattice in $G_n$ and this implies that the holonomy map is injective and its image $\Gamma \subseteq G_n$ contains no (non-trivial) elliptic elements. The same is then true of the projection $\overline{\Gamma} \subseteq \textit{PSL}(2,\mathbb{R})$ of $\Gamma$ and it follows from standard results about subgroups of $\textit{PSL}(2,\mathbb{R})$ that $\overline{\Gamma}$ is either discrete or \emph{elementary} (i.e.\ it has a finite orbit when acting on the closure of the Poincar\'{e} disc). If the subgroup $\overline{\Gamma}$ were elementary, then after taking a finite index subgroup the action would fix a point on the boundary of the disc. Thus (up to conjugacy) $\overline{\Gamma}$ is contained in the subgroup of upper triangular matrices and is hence solvable, which is absurd. It then follows that $\overline{\Gamma}$, and hence $\pi_1(\Sigma_g) \cong \Gamma$, is a discrete subgroup, which necessarily acts co-compactly and the classification follows.
\end{proof}
\noindent It is customary to call foliations that are given associated to a left-invariant Anosov flow on a left quotient of a Lie group by a co-compact lattice  \emph{algebraic Anosov} and in the sequel we adopt this terminology. 
\begin{rem}\label{torsion_lattice}
We note that Theorem \ref{Ghys_top} readily generalises to horizontal foliations on an arbitrary Seifert fibered space $M$ yielding a similar classification result. In this situation, one can consider the holonomy map associated to the foliation as a representation of the orbifold fundamental group $\pi^{orb}_1(B)$ of the base (which is necessarily hyperbolic), by considering the holonomy around curves in the base that avoid the orbifold points. Assuming that $\mathcal{F}$ is topologically conjugate to a smooth Anosov foliation one can alter the smooth structure exactly as above so that the flow becomes smooth and then the arguments of (\cite{Ghy} Section 5) sketched above provide a transverse projective structure. Thus, just as in the case of an $S^1$-bundle, the holonomy representation is smoothly conjugate to one with values in some $G_n$. One then applies Theorem \ref{Ghys_proj_2} and argues exactly as above, after first taking a torsion free subgroup of finite index, to deduce that the resulting holonomy map must then be given by the inclusion of a co-compact lattice in $G_n$.
\end{rem}
As a consequence of the special structure of Anosov flows on $S^1$-bundles, the rotation numbers of an Anosov representation are stable under deformations. We are grateful to S.\ Matsumoto for suggesting a simplified proof of the following lemma.
\begin{lem}\label{top_stab_An}
Let $\rho_{An} \in \text{Rep}_e(\pi_1(\Sigma_g),\text{Diff}_+(S^1))$ be an Anosov representation. Then for any representation $\rho$ that is sufficiently $C^0$-close to $\rho_{An}$ the rotation numbers of $\rho(\gamma)$ are rational for all $\gamma \in \pi_1(\Sigma_g)$.
\end{lem}
\begin{proof}
After lifting $\rho_{An}$ to $\widetilde{\text{Rep}}_e(\pi_1(\Sigma_g),\widetilde{\text{Diff}}_+(S^1))$ we consider the foliation $\mathcal{F}$ given by Lemma \ref{holonomy_foliation}. Let $\mathcal{F} = \mathcal{F}_s, \mathcal{F}_u$ be the weak stable and unstable foliations of the Anosov flow generated by the normalised vector field $X$ generating $\mathcal{F}_u \cap \mathcal{F}_s$. For $\rho$ sufficiently $C^0$-close to $\rho_{An}$ we choose lifts to $\widetilde{\text{Rep}}_e(\pi_1(\Sigma_g),\widetilde{\text{Diff}}_+(S^1))$ that are also $C^0$-close. Then the tangent distributions of the associated foliation $T\mathcal{F}_{\rho}$ remains $C^0$-close to $T \mathcal{F}_s$ (cf.\ Corollary \ref{continuity}) and hence the foliation $\mathcal{F}_\rho$ remains transverse to $\mathcal{F}_u$, since transversality is an open condition. Moreover, the normalised vector field $X'$ generating the intersection $\mathcal{F}_{\rho} \cap \mathcal{F}_{u}$ is $C^0$-close to $X$. 

By Theorem \ref{Ghys_top} the foliation $ \mathcal{F}_s$ is smoothly conjugate to an algebraic Anosov foliation given by the suspension of a lattice in some covering of $\textit{PSL}(2,\mathbb{R})$. Hence we can assume that the flow $\Phi_t^X$ is conjugate to a covering of the geodesic flow on the unit tangent bundle $ST\Sigma_g$ for some choice of hyperbolic metric on $\Sigma_g=\mathbb{H}^2/ \Gamma$. We consider the geodesic flow on the unit tangent bundle of $\mathbb{H}^2$ as well as its associated weak unstable and stable foliations $\widetilde{\mathcal{F}}_u, \widetilde{\mathcal{F}}_s$. Note that the projection to the base induces an isometry on the leaves of $\widetilde{\mathcal{F}}_u$ and $\widetilde{\mathcal{F}}_s$ respectively. These coverings fit into the following commutative diagram:
$$\xymatrix{ \pi^*ST\mathbb{H}^2 \ar[r] \ar[d]& M(e)\ar[d]^{\pi}\\
ST\mathbb{H}^2 \ar[d] \ar[r]& ST\Sigma_g \ar[d]\\ 
\mathbb{H}^2 \ar[r] &\Sigma_g = \mathbb{H}^2/ \Gamma.
}$$
The induced flow on a leaf $L$ of the weak unstable foliation $\widetilde{\mathcal{F}}_u$ consists of geodesics that are all asymptotic in forward time to the same point on $\partial_{\infty} \mathbb{H}^2$. In particular, the angle between the flow lines on $L$ and the boundary of an $\epsilon$-neighbourhood $N_{\epsilon}(\sigma)$ with respect to the hyperbolic metric of \emph{any} (geodesic) flow line $\sigma$ on $L$ is constant and the flow points into $N_{\epsilon}(\sigma)$ along the boundary. We consider an isometric lift $\overline{L}$ of $L$ to $\pi^*ST\mathbb{H}^2$. Assume that $\sigma$ lies on a leaf $\overline{L}$ and projects to a closed lift $\overline{\gamma}$ in $M(e)$ of a $k$-fold cover of a simple closed geodesic $\gamma$ in $ST\Sigma_g$. Note that in this case the image of the leaf $\overline{L}$ in $M(e)$ is a cylinder. We factor the covering map of $M(e)$ above, by first taking the cyclic covering corresponding to the element $[\gamma] \in \Sigma_g$:
$$\xymatrix{ \overline{L} \subseteq\pi^*ST\mathbb{H}^2 \ar[r] \ar[d]&\overline{M}(e) = \pi^*ST\mathbb{H}^2/\mathbb{Z} \ar[r] \ar[d] & M(e)\ar[d]\\
\mathbb{H}^2 \ar[r] &\mathbb{H}^2/\mathbb{Z} \ar[r] &\Sigma_g = \mathbb{H}^2/ \Gamma.}$$
Then the image of $N_{\epsilon}(\sigma)$ is an annular neighbourhood $\overline{N}_{\epsilon} = N_{\epsilon}(\sigma)/ \mathbb{Z}$ of a lift $\gamma'$ that projects \emph{diffeomorphically} to $\overline{\gamma}$, so that the vector field $X$ has angle uniformly bounded away from $0$ along $\partial \overline{N}_{\epsilon}$. Since $X'$ is $C^0$-close to $X$ on $M(e)$ the same is true for the lifted vector field on the covering $\overline{M}(e)$. Thus the first return map of the flow on $\overline{N}_{\epsilon}$ is a contraction so that $X'$ also has a closed orbit in $\overline{N}_{\epsilon}$, and this closed orbit is isotopic to $\gamma'$. It follows that the rotation number of the holonomy given by the $k$-th iterate $\gamma^k$ is $1$ and thus $\textrm{rot}(\rho(\gamma))$ must be rational.

Hence the lemma holds for holonomies around simple closed geodesics in the base. However, every free homotopy class of loops in $\Sigma_g$ has a simple geodesic representative after passing to some finite cover. We observe that the arguments are not affected by taking finite covers. Thus we can assume that every element in $\pi_1(\Sigma_g)$ is conjugate to one that is represented by a simple closed geodesic and since the rotation number is conjugation invariant the lemma follows.
\end{proof}
In order to exploit the topological stability of Anosov foliations, we will need a result of Matsumoto that characterises the conjugacy type of a representation in terms of translation numbers (cf.\ \cite{Mat0} Theorem 1.1).
\begin{lem}\label{Matsumoto_Lemma}
Let $\rho_1,\rho_2 \in \text{Rep}\thinspace (\Gamma,\text{Homeo}_+(S^1))$ for an arbitrary finitely generated group $\Gamma$ and assume for all $g \in \Gamma$ there are lifts $\widetilde{\rho_1(g)},\widetilde{\rho_2(g)}$ to $\widetilde{\text{Homeo}}_+(S^1)$ such that the translation numbers satisfy
$$tr(\widetilde{\rho_1(g_1)}\widetilde{\rho_1(g_2)})-tr(\widetilde{\rho_1(g_1)}) -tr(\widetilde{\rho_1(g_2)})= tr(\widetilde{\rho_2(g_1)}\widetilde{\rho_2(g_2)})-tr(\widetilde{\rho_2(g_1)}) -tr(\widetilde{\rho_2(g_2)})$$
for all $g_1,g_2 \in \Gamma$. If in addition $rot(\rho_1(s_k)) = rot(\rho_2(s_k))$ for some generating set $\langle s_k \rangle \subset G$, then $\rho_1$ and $\rho_2$ have the same bounded integral Euler class and the representations are semi-conjugate. If the actions are minimal, then they are in fact conjugate.
\end{lem}

\noindent With the aid of this lemma we obtain the following theorem, which answers a question posed to us by Y.\ Mitsumatsu. For the statement we let $a_i,b_i \in \pi_1(\Sigma_g)$ be standard generators for the fundamental group.

\begin{thm}\label{open_closed_Anosov}
Let $\rho_t$ be a $C^0$-continuous path in $\text{Rep}(\pi_1(\Sigma_g),\text{Diff}_+(S^1))$ such that $\rho_0$ is Anosov. Then $\rho_t$ consists entirely of Anosov representations.

Moreover, the space of Anosov representations $Rep_{An} \subset \text{Rep}(\pi_1(\Sigma_g),\text{Diff}_+(S^1))$ has finitely many $C^0$-path components and the map
$$Rep_{An} \longrightarrow (\mathbb{Z}_k)^{2g} \quad , \quad \rho \longmapsto (\text{rot}(\rho(a_i)),\text{rot}(\rho(b_i)))_{i=1}^{g}$$
induces a bijection of path components, where $k$ is such that the Euler number of the underlying $S^1$-bundles satisfies $k \thinspace e = 2g-2$. 
\end{thm}
\begin{proof}
Let $\rho_t$ be a $C^0$-continuous path starting at an Anosov representation $\rho_0$. We set
$$S_{An} = \{t \ | \  \rho_{\tau} \text{ is Anosov for all } 0 \le \tau \le t \}.$$
We first show that $S_{An}$ is open. Assume that $\rho_{s}$ is Anosov for all $\tau \le s \in [0,1]$. By Lemma \ref{top_stab_An} the rotation numbers of $\rho_{t}(g)$ are rational for any $g \in \pi_1(\Sigma_g)$ and $t$ sufficiently close to $s$ (independently of $g$). Thus by the continuity of the rotation number, these rotation numbers must be constant for $t$ close to $s$. We choose a lift $\widetilde{\rho_{t}(g)}$ of the path $\rho_{t}(g)$ to $\widetilde{\text{Diff}}_+(S^1)$ for each $g\in \pi_1(\Sigma_g)$. Then $\widetilde{\rho_{t}(g)}$ has the same translation number as $\widetilde{\rho_{s}(g)}$ for all elements $g$. It follows that 
$$tr(\widetilde{\rho_{t}(g_1)}\widetilde{\rho_{t}(g_2)}) \equiv tr(\widetilde{\rho_{t}(g_1g_2)}) =  tr(\widetilde{\rho_s(g_1g_2)})\equiv tr(\widetilde{\rho_s(g_1)}\widetilde{\rho_s(g_2)}) \text{ mod } \mathbb{Z}$$
and we conclude by continuity of translation numbers that  $tr(\widetilde{\rho_{t}(g_1)}\widetilde{\rho_{t}(g_2)}))$ is constant for all $t$ sufficiently close to $s$. It follows that the hypotheses of Lemma \ref{Matsumoto_Lemma} are satisfied for $\rho_s$ and $\rho_{t}$ where $|s-t|<\epsilon$ is sufficiently small and both representations have the same bounded integral Euler class and are thus semi-conjugate. 

Furthermore, the action on $S^1$ induced by $\rho= \rho_{t}$ either has a finite orbit, an exceptional minimal set or it is minimal. Note that the first case is ruled out since the Euler class is non-zero. If the action had an exceptional minimal set $K \subset S^1$, then the semi-proper leaves of the associated suspension foliation must have infinitely many ends by Duminy's Theorem (cf.\ \cite{CaCo}). Moreover, since the actions are semi-conjugate, we conclude that the Anosov foliation given by the suspension of $\rho_{s}$ also has a leaf with infinitely many ends. But the weak stable foliation of an Anosov flow can have only leaves with $1$ or $2$ ends yielding a contradiction. Thus in fact $\rho$ itself must be minimal and hence again by Lemma \ref{Matsumoto_Lemma} it is topologically conjugate to $\rho_{s}$. Finally by Theorem \ref{Ghys_top} the representation $\rho = \rho_{t}$ is then smoothly conjugate to an (algebraic) Anosov representation and the openness of the set $S_{An}$ follows.

We next show that $S_{An}$ is also closed. Let $(t_n)$ be a sequence of elements in $S_{An}$ converging to $s$. By construction we can assume that this convergence is monotone, i.e.\ $t_n \nearrow s$. Now by the arguments above the rotation numbers of $\rho_t(g)$ are constant for each $g\in \pi_1(\Sigma_g)$ and $t < s$ and it follows that $\rho_t(g)$ is also constant on the interval $[0,s]$. Taking lifts it again follows that:
$$tr(\widetilde{\rho_{t}(g_1)}\widetilde{\rho_{t}(g_2)}) \equiv tr(\widetilde{\rho_s(g_1)}\widetilde{\rho_s(g_2)})  \text{ mod } \mathbb{Z} \text{, for all } t \in[0,s].$$
Then arguing exactly as above we deduce that $\rho_{s}$ is topologically, and hence smoothly, conjugate to an Anosov representation. Consequently the set
$S_{An}$ is both open and closed and it is obviously non-empty. We conclude that the path $\rho_t$ is wholly contained in $Rep_{An}$, showing that this set indeed consists of $C^0$-path components of $\text{Rep}(\pi_1(\Sigma_g),\text{Diff}_+(S^1))$.

Due to Theorem \ref{Ghys_top} and results of Goldman \cite{Gol} we can now describe the $C^0$-path components of the space of Anosov representations explicitly. Any Anosov representation is smoothly conjugate to an embedding of a discrete subgroup in the $k$-fold cover $G_k$ of $\textit{PSL}(2,\mathbb{R})$ by Theorem \ref{Ghys_top}, where $k$ is determined by the Euler class of the representation. The number of components of $\text{Rep}_{max}(\pi_1(\Sigma_g),G_k)$ is finite by \cite{Gol} and are distinguished by elements in $H^1(\pi_1(\Sigma_g), \mathbb{Z}_k)$. In particular, the path components of $\text{Rep}_{max}(\pi_1(\Sigma_g),G_k)$ can be distinguished by the rotation numbers on the images of  generators $a_i,b_i$. To see this note that the rotation number of an algebraic Anosov representation lies in the $k$-th roots of unity $\mathbb{Z}_k \subset S^1$ so that the rotation numbers $\rho_{An}(a_i),\rho_{An}(b_i)$ are constant on $C^0$-components. This concludes the proof, since the maps $\{a_1,b_1,...,a_g,b_g \} \to \mathbb{Z}_k$ given by $\textrm{rot}(\rho_{An})$ correspond precisely to the elements $H^1(\pi_1(\Sigma_g), \mathbb{Z}_k)$ that distinguish components of $\text{Rep}_{max}(\pi_1(\Sigma_g),G_k)$.
\end{proof}
\noindent As a consequence of Theorem \ref{open_closed_Anosov} we obtain the following extension of Ghys and Matsumoto's global stability statement about conjugacy classes of representations in $\text{Rep}(\pi_1(\Sigma_g),\text{Diff}_+(S^1))$ for the case of maximal Euler class \cite{Ghy}, \cite{Mat} to other topological components. 
\begin{cor}\label{Anosov_comp}
Any representation $\rho \in \text{Rep}(\pi_1(\Sigma_g),\text{Diff}_+(S^1))$ that lies in the $C^0$-path component of an Anosov representation $\rho_{An}$ is smoothly conjugate to an embedding of a discrete subgroup in some finite cover of $\textit{PSL}(2,\mathbb{R})$ and is topologically conjugate to $\rho_{An}$. In particular, it is an injective discrete co-compact representation.
\end{cor}
\begin{rem}
Theorem \ref{open_closed_Anosov} and its corollary also remain true for Anosov representations of any hyperbolic orbifold group $\pi_1^{orb}(B_{hyp})$ in view of Remark \ref{torsion_lattice}. Since both Theorem \ref{open_closed_Anosov} and its corollary hold with respect to the $C^0$-topology, they yield a proof of the $C^0$-version of Theorem \ref{Rep_disconnected} without using contact topology in the form of Vogel's results \cite{Vog}.
\end{rem}
\begin{rem}
The arguments used in the proof of Theorem \ref{open_closed_Anosov} imply that Corollary \ref{Reebless_comp} holds for any Anosov foliation on a Seifert fibered space. One can use the closedness of the space of Anosov representations instead of the results of Matsumoto \cite{Mat} to rule out the first instance of unstable stacks of torus leaves and the remainder of the proof holds $verbatim$.
\end{rem}
If we restrict ourselves to the $C^1$-topology then we obtain a slight strengthening of Theorem \ref{open_closed_Anosov} which then distinguishes connected components rather than just path components.
\begin{thm}\label{open_closed_Anosov_C_1}
The space of Anosov representations $Rep_{An} \subset \text{Rep}(\pi_1(\Sigma_g),\text{Diff}_+(S^1))$ is both open and closed with respect to the $C^{1}$-topology. It has finitely many connected components which are distinguished by the rotation numbers of the images of a set of standard generators $a_i,b_i \in \pi_1(\Sigma_g)$.
\end{thm}
\begin{proof}
The openness follows immediately from the $C^1$-stability of the Anosov condition. Closedness follows from the fact that any sequence of Anosov representations $(\rho_n)$ converging to some $\rho_0$ can (after possibly passing to a subsequence) be assumed to all be topologically conjugate to a \emph{fixed} Anosov representation by Theorem \ref{Ghys_top} and Matsumoto \cite{Mat}. Thus rotation numbers and translation numbers of lifts are automatically constant and one can argue exactly as in the proof of Theorem \ref{open_closed_Anosov} to deduce closedness.
\end{proof}
It would be interesting to know whether the above results also hold for (non-smooth) topological actions on $S^1$, which would involve showing some sort of stability statements for laminations that are Anosov in some suitable sense. In fact, by results of Ghys, topological actions on $S^1$ are classified up to semi-conjugacy by the bounded integral Euler class $e^{\mathbb{Z}}_b$. It then seems to be an open problem to determine the image
$$\textrm{Rep}(\pi_1(\Sigma_g),Homeo_+(S^1)) \stackrel{e^{\mathbb{Z}}_b} \longrightarrow H_b^2(\pi_1(\Sigma_g),\mathbb{Z}).$$
Note that the map to \emph{real bounded cohomology} given by the real bounded Euler class $e^{\mathbb{R}}_b$ is continuous with respect to the weak-$^*$ topology on the vector space $H_b^2(\pi_1(\Sigma_g),\mathbb{R})$. One also has a natural map
$$\textrm{Rep}(\pi_1(\Sigma_g),Homeo_+(S^1)) \stackrel{\Phi_g} \longrightarrow (S^1)^{2g} = T^{2g}$$
given by the rotation numbers on the standard generators of $\pi_1(\Sigma_g)$. Since the the bounded integral Euler class of an action on $S^1$ is determined by its real Euler class together with the rotation numbers on generators by Matsumoto \cite{Mat0}, one can identify the image of $e^{\mathbb{R}}_b \times \Phi_g$ with the image of $e^{\mathbb{Z}}_b$, which then inherits a topology in a natural way. It would be interesting to understand whether this image is connected or not, with respect to the weak-$^*$ topology, which would then provide insights into the topology of the representation spaces in which we are interested. In fact one knows that the image consists of classes that admit cocycle representatives taking only the values $0,1$. The straight line between any two such cocycles $tz_0 + (1-t)z_1$ obviously gives a continuous path in $H_b^2(\pi_1(\Sigma_g),\mathbb{R})$. However, it is not clear, and perhaps very unlikely, that this path lies in the image of the bounded Euler class.

\end{document}